\documentclass[12pt]{amsart}
\usepackage{amsmath,latexsym,amsfonts,amssymb,amsthm}
\usepackage{geometry}
\usepackage{mathrsfs}
\usepackage{graphicx,color}
\usepackage{tikz-cd}
\usepackage[all,cmtip]{xy}
\usepackage{enumitem}
\usepackage[colorlinks, citecolor=blue]{hyperref}

\newtheorem{prop}{Proposition}[section]
\newtheorem{thm}[prop]{Theorem}
\newtheorem{cor}[prop]{Corollary}
\newtheorem{conj}[prop]{Conjecture}
\newtheorem{lem}[prop]{Lemma}

\theoremstyle{definition}
\newtheorem{que}[prop]{Question}
\newtheorem{defn}[prop]{Definition}
\newtheorem{expl}[prop]{Example}
\newtheorem{rem}[prop]{\it Remark}

\numberwithin{equation}{section}

\newcommand{\bP}{\mathbb{P}}
\newcommand{\bC}{\mathbb{C}}

\newcommand{\bA}{\mathbb{A}}
\newcommand{\bQ}{\mathbb{Q}}
\newcommand{\bZ}{\mathbb{Z}}

\newcommand{\cX}{\mathcal{X}}
\newcommand{\cY}{\mathcal{Y}}

\newcommand{\cO}{\mathcal{O}}
\newcommand{\cL}{\mathcal{L}}
\newcommand{\cI}{\mathcal{I}}

\newcommand{\cF}{\mathcal{F}}
\newcommand{\cG}{\mathcal{G}}
\newcommand{\cJ}{\mathcal{J}}

\newcommand{\cN}{\mathcal{N}}

\newcommand{\cT}{\mathcal{T}}

\newcommand{\fa}{\mathfrak{a}}
\newcommand{\fb}{\mathfrak{b}}
\newcommand{\fc}{\mathfrak{c}}
\newcommand{\fm}{\mathfrak{m}}

\newcommand{\Spec}{\mathrm{Spec}}
\newcommand{\Supp}{\mathrm{Supp}}
\newcommand{\lct}{\mathrm{lct}}
\newcommand{\mult}{\mathrm{mult}}
\newcommand{\Ext}{\mathrm{Ext}}

\newcommand{\rom}[1]{\lowercase\expandafter{\romannumeral #1\relax}}
\newcommand{\vol}{\mathrm{vol}}
\newcommand{\hvol}{\widehat{\vol}}
\newcommand{\Val}{\mathrm{Val}}
\newcommand{\ord}{\mathrm{ord}}

\newcommand{\Sym}{\mathrm{Sym}}
\newcommand{\Rep}{\mathrm{Rep}}
\newcommand{\id}{\mathrm{id}}
\newcommand{\Jac}{\mathrm{Jac}}

\newcommand{\Tr}{\mathrm{Tr}}
\newcommand{\edim}{\mathrm{edim}}
\newcommand{\Bl}{\mathrm{Bl}}

\newcommand{\pr}{\mathrm{pr}}

\newcommand{\Ex}{\mathrm{Ex}}
\newcommand{\Link}{\mathrm{Link}}
\newcommand{\Aut}{\mathrm{Aut}}
\newcommand{\locpione}{\hat{\pi}^{\mathrm{loc}}_1}

\newcommand\numberthis{\addtocounter{equation}{1}\tag{\theequation}}

\begin{document}

\title{K-stability of cubic threefolds}

\abstract{We prove the K-moduli space of cubic threefolds is identical to their GIT moduli. More precisely, the K-(semi,poly)-stability of cubic threefolds coincide to the corresponding GIT stabilities, which could be explicitly calculated. In particular, this implies that all smooth cubic threefolds admit K\"ahler-Einstein metric as well as provides a precise list of singular KE ones. To achieve this, our main new contribution is an estimate in dimension three of the volumes of kawamata log terminal singularities introduced by Chi Li. This is obtained via a detailed study of the classification of three dimensional canonical and terminal singularities, which was established during the study of the explicit three dimensional minimal model program.} 
}
\date{\today}

\author{Yuchen Liu}
\address{Department of Mathematics, Princeton University, Princeton, NJ 08544-1000, USA}
\email{yuchenl@math.princeton.edu}

\author{Chenyang Xu}
\address{Beijing International Center for Mathematical Research, Beijing 100871, China}
\email{cyxu@math.pku.edu.cn}

\maketitle
\tableofcontents

\date{\today}

\maketitle

\section{Introduction}
After the celebrated work of \cite{CDS15} and \cite{Tia15}, we know that a Fano manifold has a K\"ahler-Einstein (KE) metric if and only if it is K-polystable. Then the main question left for the existence of KE metric on a Fano manifold is how to check its K-polystability. A classical strategy first appeared in \cite{Tia92} was using deformation in the parametrizing space, so that from one Fano manifold $X$ known to have KE metric, we can use continuity method to study other Fano manifolds which can deform to $X$. This idea is successfully used to find out all smooth del Pezzo surfaces with a KE metric in \cite{Tia92} and then extended to all (not necessarily smooth) limits of quartic del Pezzo surfaces in \cite{MM93} which also gives an explicit construction of the compact moduli space. Later with a more focus on the stability study,  the work of \cite{MM93} was further extended to limits of all smooth KE surfaces in \cite{OSS16}.

The strategy can be summarized by two steps:  in the first step, we need to give a good control of all the possible local singularities appeared on the limit by bounding their local volume; then in the second step, we show such limits can all be embedded in an explicit ambient space, and this often leads to an explicit characterization by more standard methods, e.g. the geometric invariant theory (GIT). 

After the case of surface is completely settled, it is natural to  apply this strategy to higher dimensional examples. Built on the results in \cite{CDS15, Tia15, Ber16}, in \cite{LWX14} (see also \cite{Oda15, SSY16}) we construct  an algebraic scheme $M$ which is a good quotient moduli space with closed points parametrizing all smoothable K-polystable $\mathbb{Q}$-Fano varieties $X$. However, the construction is essentially theoretical and can not lead to an effective calculation. One sticky point is that although $X$ is often explicitly given, the limits from the continuity method may be embedded in a much larger ambient space, which we do not have a direct control of it. 

\bigskip

On the other hand, there is an indirect way to study the limit. In fact, using a completely algebraic approach and built on the global work of \cite{Fuj15} and the local study of \cite{Li15a}, in \cite{Liu16} an inequality 
$$\hvol(x,X)\cdot \left(\frac{n+1}{n}\right)^n\ge (-K_X)^n$$
between the global volume of an $n$-dimensional K-semistable Fano variety $X$ and the local volume of each singularity $x\in X$ is established. Here we note that following \cite{Li15a}, the local volume $\hvol(x,X)$ of a kawamata log terminal (klt) singularity $x\in X$ is defined as the minimal normalized volume $\hvol_{x,X}(v)$ for all valuations $v$ in ${\rm Val}_{x,X}$. See \cite{Li15b, LL19, LX16, Blu16} for some recent progress on this topic. 

Then when the volume of $X$ is large, by a detailed analysis of volumes of singularities,  we hope that  the local volume bound obtained by \cite{Liu16} is restrictive enough so that we can use it to show that all the limiting objects are only mildly singular. Once this is true, then we will have a chance to proceed as in the surface case by showing that $X$ and its limits are contained in a (not very large) natural ambient space. This way we could obtain the needed explicit description. 

\bigskip

In this note, we carry out this strategy for cubic threefolds and prove the following theorem. 
\begin{thm}\label{t-global}
 If a (possibly singular) cubic hypersurface $X\subset \mathbb{P}^4$ is GIT polystable (resp. semistable), then $X$ is K-polystable (resp. K-semistable).
 
  In particular, if we let $(U^{\rm ss}\subset \mathbb{P}^{34})$ parametrize all GIT-semistable Fano cubic hypersurfaces in $\mathbb{P}^4$, the GIT quotient morphism
$$U^{\rm ss}\to M^{\rm GIT}:=_{\rm defn}U^{\rm ss}/\!/{\rm PGL(5)}.$$ explicitly yields the proper good quotient moduli space  $M$  parametrizing all  K-polystable threefolds which can be smoothable to a cubic threefold.  
\end{thm}

By the classification in \cite{All03}, we have a concrete description. 
\begin{cor}\label{c-list}
We have the following list which gives all the closed points of $M(=M^{\rm GIT})$ in Theorem \ref{t-global}:
\begin{enumerate}
\item All smooth cubic threefolds are K-stable;
\item  All cubic threefolds only containing isolated $A_k$ singularities $(k\le 4)$ are K-stable;
\item There are two type K-polystable cubic threefolds with non-discrete automorphisms: 
$$F_{\Delta}=x_0x_1x_2+x_3^3+x_4^3 \qquad{\mbox and \ \ }F_{A,B}=Ax_2^3+x_0x_3^2+x_1^2x_4-x_0x_2x_4+Bx_1x_2x_3.$$
In particular, each cubic threefold on the above list admits a KE metric. 
\end{enumerate}
\end{cor}

We note that combined Corollary \ref{c-list} with \cite{Che01, CP02}
and \cite[Corollary 1.4]{Fuj16} which say that all smooth quartic threefolds are K-stable, we answer affirmatively the folklore conjecture that all  smooth Fano hypersurfaces have KE metrics for dimension 3. 

\bigskip

As we mentioned, we need a local result which uses the volume to bound the singularities. For instance, we aim to show that all objects parametrized by $M$ are Gorenstein.  The key local result is the following.
\begin{thm}\label{t-local}
Let $x\in X$ be a three dimensional (non-smooth) klt singularity. Then
\begin{enumerate}
\item  $\hvol(x,X)\le 16$ and the equality holds if and only if it is an $A_1$ singularity;
\item If $x\in X$ is a quotient singularity by a finite group $G$ whose action
is free in codimension $1$, then $\hvol(x,X)=27/|G|$;
\item $\hvol(x,X)\le \ \frac {27}{r}$ where $r$ is the maximal order of a torsion element in the class group;
\item If $x\in X$ is not a quotient singularity and there exists a nontrivial torsion class in $\mathrm{Pic}(x\in X)$, then $\hvol(x,X)\leq 9$.
\end{enumerate}
\end{thm}

\begin{rem}\label{rem-ss}
While we were preparing this note, the authors of \cite{SS17} informed us by using a more analytic approach, they made a similar investigation on $n$-dimensional del Pezzo manifolds of degree 4, i.e. the smooth intersections of two quadratics in ${\mathbb{P}}^{n+2}$.  More precisely, they obtain the description  of the K-moduli as the GIT quotient for the compactification of degree 4 del Pezzo manifolds.  In particular, this way they give a different proof of the existence of KE metric on degree 4 del Pezzo manifolds which was first established in \cite{AGP06}.

In fact, they  also consider cubic hypersurfaces in \cite{SS17}, and show that the statement as in Theorem \ref{t-global} saying that the GIT stability is the same as K-stability in any arbitrary dimension follows from an estimate of the largest volume of non-smooth points as stated in Conjecture \ref{c-largev} (but for an analytic definition of the local volume). 
So the main technical contribution in the current paper is the various local estimates in Theorem \ref{t-local}. For instance, Theorem \ref{t-local}.1 settles Conjecture \ref{c-largev} in dimension three.
\end{rem}


\bigskip

It was shown in \cite{TX17} that the local fundamental group of a three dimensional algebraic klt singularity is finite (also see \cite{Xu14} for a result on general dimension). Using similar arguments in proving Theorem \ref{t-local}, we give an effective upper bound on the size of the local fundamental group in terms of volumes:

\begin{thm}\label{fundgp}
 If $x\in X$ is a three dimensional algebraic klt singularity, then 
 \[
 |\pi_1(\Link(x\in X))|\cdot\hvol(x,X)< 324.
 \]
\end{thm}

When a klt singularity appears on the Gromov-Hausdorff limit of
smooth KE Fano manifolds, it is easy to see the volume is at
most $n^n$ (which is the volume of smooth points) by Bishop Comparison 
Theorem. In Appendix \ref{hvolbound}, we show that this in fact holds for all klt
singularities as in the following theorem. This can be viewed 
as a local analogue of the global volume bounds in
\cite[Theorem 1.1]{Fuj15}.

\begin{thm}\label{t-hvolbound}
  Let $x\in X$ be an $n$-dimensional klt singularity.
 Then $\hvol(x,X)\leq n^n$ and the equality holds if and
 only if $x\in X$ is smooth.
\end{thm}

\subsection{Outline of the proof}
To help the reader to understand the results, we give an outline of our proofs. We first explain our strategy for proving  Theorem \ref{t-local}, which is  the most technical part of this paper. 
The proof of Theorem \ref{t-local}.1 relies on a detailed study of three dimensional canonical singularities, which is a classical topic on birational geometry of threefolds. More precisely, by taking the index-1 cover and applying Theorem \ref{ghvol}, we can reduce to the case that the singularity is Gorenstein; then by induction on the number of the crepant divisors, i.e. divisors with discrepancy 0, and applying Lemma \ref{hvolup}, we only need to understand two cases: a Gorenstein terminal singularity and a Gorenstein canonical singularity equipped with a smooth crepant resolution with an irreducible exceptional divisor. In dimension three, a Gorenstein terminal singularity is known to be a hypersurface singularity, and then the estimate is straightforward (see Lemma \ref{hypersurf}). In the case of Gorenstein canonical (but non-terminal) singularity, the geometry of the exceptional divisor has been classified (see \cite{Rei94}), and a detailed study of such a classification is used to find suitable valuations to complete the estimate (see Proposition \ref{l-weak}). In fact, we can always choose a valuation $v$ which is either a valuation over a singularity on the terminalization, hence we reduce to the case of Gorenstein terminal singularity, or a crepant divisor, i.e., a divisor with discrepancy 0,  such that $\hvol(v)\le 16$. In our argument, the classification results of three dimensional singularities are essential, and their generalizations to higher dimension seem to be challenging to us. 

To obtain Theorem \ref{t-local}.2-4, another technical point in our argument we want to make is that the assumptions in Theorem \ref{t-local}.2-4 yield quasi-\'etale Galois coverings, that is, Galois coverings which are ramified along loci of codimension at least 2. Therefore, ideally a key ingredient to prove this kind of results would be a multiplication formula as stated in Conjecture \ref{c-finite}. 
Unfortunately, for now we are still lack of a method to prove Conjecture \ref{c-finite}.  So while Theorem \ref{t-local}.2 directly follows from \cite{LX16}, we have to go through a more complicated discussion and obtain weaker results as in Theorem \ref{t-local}.3-4. For Theorem \ref{t-local}.3, when the index-1 cover is terminal, we can apply Lemma \ref{hypersurf}; and if the index 1-cover is strictly canonical, we can find apply \cite{HX09} to find a $\mathbb{Z}/r$-fixed point on the terminalization (see Lemma \ref{gorindex}). The argument for \ref{t-local}.4 is similar to the proof of  Theorem \ref{t-local}.1, but we also need to track the group action. By the previous discussions, the only case left is when the index-1 cover is strictly canonical and the cover map is of degree 2. Then we again study Reid's classification in details and carefully choose a $\mathbb{Z}/2$-invariant valuation with normalized volume no more than 18 (see Proposition \ref{p-strong}). In one case, we need to look at a new type of valuations (see the last case in the proof of Lemma \ref{l-equivirr}).

\medskip
As we mentioned, once  Theorem \ref{t-local} is proved, then the implications to Theorem \ref{t-global} and Theorem \ref{fundgp} is  straightforward. 

\medskip

Finally, the proof of Theorem \ref{t-hvolbound} uses the fact that under a specialization, we can take a flat degeneration of an $\fm$-primary ideal so that the colengths remain the same, but the log canonical thresholds only possibly decrease.

\bigskip

\noindent {\bf Notation and Conventions:} We follow the standard notations as in \cite{Laz04a, KM98, Kol13}. Any singularity $(X,x)$ we consider in this note is the localization of a point $x$ on an algebraic variety $X$ (over $\mathbb{C}$). We use the standard notation $\frac{1}{r}(a_1,...,a_n)$ to mean the quotient singularity given by $\mathbb{Z}/r$ action generated by $g\cdot (x_1,...,x_n)\to (\xi^{a_1}x_1,...,\xi^{a_n}x_n)$ where $\xi$ is a primitive $r$-th root of unity. 

If $(X,\Delta)$ is a log pair such that $K_X+\Delta$ is $\mathbb{Q}$-Cartier, we denote by $A_{(X,\Delta)}(v)$ (or $A_{X}(v)$ if $\Delta=0$) the log discrepancy of a valuation $v$ over $(X,\Delta)$ (see \cite[Theorem 3.1]{BdFFU15}).

\bigskip
\noindent{\bf Acknowledgement:} The project was initiated in a discussion with Radu Laza in 2017 AIM workshop `Stability and Moduli Spaces'.  We thank Radu Laza for the inspiring conversation and the organizers for providing us the collaborating chance. We also want to thank Valery Alexeev, Harold Blum, Chi Li, J\'anos Koll\'ar, Song Sun, Gang Tian and Ziquan Zhuang for helpful discussions and useful comments. YL was partially supported by NSF Grant DMS-1362960. CX was partially supported by The Chinese National Science Fund for Distinguished Young Scholars (11425101). 

\section{Properties of volumes}
\subsection{Volume of a klt singularity}In this section, we define the volume of a klt singularity as the minimum of the normalized volume of all real valuations over $x$. 
The latter notion of normalized volume is first defined in \cite{Li15a} as follows: 
\begin{defn}[{\cite[Section 3]{Li15a}}]
 Let $X$ be an $n$-dimensional klt singularity. Let $x\in X$ be a closed point.
 Then the \emph{normalized volume function of valuations} $\hvol_{X,x}:\Val_{X,x}\to(0,+\infty)$
 is defined as
 \[
  \hvol_{X,x}(v)=\begin{cases}
            A_{X}(v)^n\cdot\vol_{X,x}(v), & \textrm{ if }A_{X}(v)<+\infty;\\
            +\infty, & \textrm{ if }A_{X}(v)=+\infty.
           \end{cases}
 \]
 Here $A_X(v)$ means the log discrepancy of the valuation, and $\vol(v)$ is the volume. 
\end{defn}

Then we can define {\it the volume of a klt singularity $x\in X$} to be
$$\hvol(x, X)=_{\rm defn} \min \hvol(v) \ \ \mbox{for all } v\in \Val_{X,x}.$$
We remark that in \cite{Li15a} it was show the infimum of the right hand side exists as a positive number; later in \cite{Blu16}, the existence of a minimum is confirmed.

\bigskip

There are two other different ways to characterize the volume of a klt singularity. One is using ideals (or graded sequence of ideals). More precisely, we have the following.
\begin{thm}[\cite{Liu16}, \cite{Blu16}]\label{eq-ideal}
Let $(X,x)=({\rm Spec} R, \fm)$, where $R$ is a local ring essentially of finite type and $\fm$ is the maximal ideal, then
$$ \hvol({x,X})= \inf_{\fa\colon\fm\textrm{-primary}} \lct(\fa)^n  \cdot  \mult(\fa)=\min_{\fa_{\bullet}\colon\fm\textrm{-primary}}  \lct(\fa_{\bullet})^n  \cdot  \mult(\fa_{\bullet}) ,$$
where $\fa_{\bullet}$ means a graded ideal sequence.
\end{thm}
For the definition of $\lct(\fa)$ and $\lct(\fa_{\bullet})$, see \cite[9.3.14 and 11.1.22]{Laz04b}.

Recall that {\it a Koll\'ar component $S$ over a klt singularity $x\in X$} is given by a birational morphism $f\colon Y\to X$, such that $f$ is isomorphic over $X\setminus \{x\}$, $f^{-1}(x)$ is the irreducible divisor $S$ where $(Y,S)$ is plt and $-S$ is $f$-ample.
Such a morphism $f:Y\to X$ is also called a plt blow up (see \cite[Definition 2.1]{Pro00} or \cite[P 412]{Xu14}). Then the second approach of characterizing the volume is using birational models. 
\begin{thm}[\cite{LX16}] \label{eq-model}
Let $(X,x)=({\rm Spec} R, \fm)$, then
$$ \hvol({x,X})= \inf_{\textrm{model Y}}  \vol (Y/X)=\inf_{ \textrm{Koll\'ar component S}} \hvol(\ord_S).$$
(For the definition of $\vol(Y/X)$, see \cite{LX16}). 
\end{thm}

\subsection{Volumes under Galois finite morphisms}

We will study the change of the volume under a finite Galois quotient morphism. As we mentioned in Conjecture \ref{c-finite}, we expect there is a degree formula. However, what we can prove in this section is weaker.

In the below, we use the approach of ideals to treat it, as we hope it could be later generalized to positive characteristics. We can also use the approach of models to get the comparison results which we need later. See Remark \ref{r-model}.

Let $(x\in X)$ be an algebraic klt singularity over $\mathbb{C}$ with a finite group action by $G$. We define 
$$\hvol^G(x, X)=_{\rm defn}\inf_{v\in \Val_{{X},{x}}^G}  \hvol(v),$$
where $\Val_{{X},{x}}^G$ means the $G$-invariant points in $\Val_{{X},{x}}$. 
  Actually, the approach in \cite{Blu16} should be extended into this setting to show that in the above definition, the infimum is indeed a minimum. More challengingly, by the uniqueness conjecture of the minimizer (see \cite[Conjecture 6.1.2]{Li15a}), we expect the following is true.
  \begin{conj}\label{c-inva}We indeed have
  $\hvol^G(x, X)=\hvol(x, X)$. 
\end{conj}
 As we will see, this is equivalent to Conjecture \ref{c-finite}.

\bigskip
 
 Throughout this section, when a finite group $G$ acts on a Noetherian local ring $(R,\fm)$, we denote by $R^G$ the subring of $G$-invariant elements in $R$. For an ideal $\fa$ of $R$, we denote by $\fa^G:=\fa\cap R^G$. Denote by $n:=\dim R$.

\begin{lem}\label{l-ringext}
 Let $(R,\fm)$ be the local ring of a complex klt singularity. Let $G\subset\Aut(R/\bC)$ be a finite subgroup acting freely in codimension $1$ on $\Spec(R)$. Then for
 any $\fm^G$-primary ideal $\fb$ in $R^G$, we have
 \[
  \lct(\fb R)=\lct(\fb),\qquad \mult(\fb R)=|G|\cdot\mult(\fb).
 \]
\end{lem}

\begin{proof}
 The first equality of $\lct$'s is an easy consequence of  \cite[5.20]{KM98}. The second equality of multiplicities follows from \cite[Theorem 14.8]{Mat86}.
\end{proof}

The following lemma is also true in  characteristic $p>0$ when the order of $G$ is not divisible by $p$. A characteristic free proof will follow from Lemma \ref{l-ringext} and \cite{Sym00}. Here we present a proof that works only in characteristic zero.

\begin{lem}\label{finite}
  Let $(R,\fm)$ be the local ring of a complex klt singularity. Let $G\subset\Aut(R/\bC)$ be a finite subgroup acting freely in codimension $1$ on $\Spec(R)$.
  Then for any $G$-invariant $\fm$-primary ideal $\fa$ of $R$, we have
 \[
  \lct(\fa)^n\mult(\fa)\geq |G|\inf_{\fb\colon\fm^G\textrm{-primary}}
  \lct(\fb)^n\mult(\fb).
 \]
\end{lem}

\begin{proof}
 Let us introduce some notations before we start the proof.
 For any local ring $(R,\fm)=(\cO_{x,X},\fm_x)$ of a  
 closed point $x$ on a normal klt variety $X$ over $\bC$, we denote by $\Jac_R$ the ideal in $R$
 as the localization of the Jacobian ideal $\Jac_X$ at $x$.
 Recall the Jacobian ideal is defined as $\Jac_X:=\mathrm{Fitt}_0(\Omega_X)$, where
 $\mathrm{Fitt}_0$ denotes the $0$-th fitting ideal as in \cite[Section 20.2]{Eis95}.
 Note that the singular locus of $X$ is equal to $\mathrm{Cosupp}(\Jac_X)$.
 If $\fa_\bullet$ is a graded sequence of $\fm$-primary ideals, then
 the asymptotic multiplier ideal $\cJ(R, m\cdot\fa_{\bullet})$ 
 with $m\in \bZ_{>0}$ is defined
 to be the maximal element of the multiplier ideals 
 $\{\cJ(R, \frac{1}{l}\cdot\fa_{ml})\}_{l\in\bZ_{>0}}$
 (see e.g. \cite{BdFFU15} or \cite[Section 3.2]{Blu16}).
 If $\fa$ is an ideal of $R$, then $\overline{\fa}$ denotes the 
 integral closure of $\fa$ in $R$.
 
 Let $d:=|G|$.
 For any element $z\in\Jac_{R^G}\cdot\fa^m$, we know that $z$
 is a root of the monic polynomial 
$$f(x):=\prod_{g\in G}(x-g(z))=x^d+c_1 x^{d-1}+c_2 x^{d-2}+\cdots+c_d.$$
Since $\Jac_{R^G}$ is $G$-invariant, we see that if we write $z=\sum^k_{j=1} \phi_j\psi_j$ where $\phi_j\in \Jac_{R^G}$ and $\psi_j\in \fa^m$,
then 
$$c_i=(-1)^i\sum_{|I_1|+\cdots+|I_k|=i}
\prod_{j=1}^k\phi^{|I_j|}_j\cdot \left(\prod_{g\in {I_{1}}}
g(\psi_1)\cdots \prod_{g \in {I_{k}}}g(\psi_k)\right), $$ where the sum runs over all $I_1,..., I_k$ which are disjoint subsets of $G$. This implies $c_i\in \Jac_{R^G}^i\cdot (\fa^{mi})^G$.  Denote by $\fb_s:=(\fa^{s})^G$. By 
 \cite[Proposition 3.4]{Blu16}, we have
 \[
  c_i\in\Jac_{R^G}^i\cdot
 (\fa^{mi})^G=\Jac_{R^G}^i\cdot\fb_{mi}\subset \cJ(R^G, m\cdot
 \fb_{\bullet})^i.
 \]
 Hence we know that $z\in\overline{\cJ(R^G,m\cdot
 \fb_{\bullet})R}$, which implies $\Jac_{R^G}\cdot\fa^m
 \subset \overline{\cJ(R^G, m\cdot
 \fb_{\bullet})R}$. Since $R$ is a finite $R^G$-module,
 we know that $\fb_1=\fa^G$ is an $\fm^G$-primary ideal in $R^G$.
 Hence $\fb_1 R$ is an $\fm$-primary ideal in $R$.  
 Choose a positive integer $l$ such that $\fm^l\subset\fb_1 R$.
 Then for any $m\in \bZ_{>0}$ we have 
 \[
  \fm^{ml}\subset (\fb_1 R)^m\subset \fb_m R
  \subset \cJ(R^G, \fb_m)R\subset\cJ(R^G, m\cdot
 \fb_{\bullet})R.
 \]
 Therefore, we have 
 $$(\Jac_{R^G}R +\fm^{ml})\cdot\fa^m
 \subset \Jac_{R^G}\cdot\fa^m+\fm^{ml}\subset\overline{\cJ(R^G, m\cdot
 \fb_{\bullet})R}.$$
 Hence by Teissier's Minkowski inequality \cite{Tei77},
 we know that
 \begin{align}
  \limsup_{m\to\infty}\frac{1}{m^n}\mult(\cJ(R^G,m\cdot
 \fb_{\bullet})R)&= \limsup_{m\to\infty}\frac{1}{m^n}
 \mult(\overline{\cJ(R^G,m\cdot \fb_{\bullet})R})\nonumber
 \\&\leq \limsup_{m\to\infty}\frac{1}{m^n}
 \mult((\Jac_{R^G}R +\fm^{ml})\cdot\fa^m)
 \nonumber \\ &\leq
 \limsup_{m\to\infty}\frac{1}{m^n}\left(\mult(\Jac_{R^G}R +\fm^{ml})^{1/n}
 +\mult(\fa^m)^{1/n}\right)^n\label{ineq1} \\
 & =\limsup_{m\to\infty}\frac{\mult(\fa^m)}{m^n}=\mult(\fa)\nonumber.
 \end{align}
 Let us explain the second last equality. Denote by 
 $R':=R/\Jac_{R^G} R$ and $\fm':=\fm/\Jac_{R^G} R$. 
 Since $\Jac_{R^G}\neq 0$, the local ring $(R',\fm')$ 
 has dimension at most $(n-1)$.
 Thus  
 \[
 \lim_{m\to\infty}\frac{\mult(\Jac_{R^G} R+\fm^{ml})}{m^n}
=\lim_{m\to\infty}\frac{\ell(R/(\Jac_{R^G} R+\fm^{ml}))}{m^n/n!}
=\lim_{m\to\infty}\frac{\ell(R'/(\fm')^{ml})}{m^n/n!}.
 \]
 The last limit is zero since $\ell(R'/(\fm')^{ml})=O(m^{\dim(R')})$
 and $\dim(R')< n$.
  
 To bound the log canonical threshold, we notice the following
 inequality appeared in \cite[3.6 and 3.7]{Mus02}:
 \[
  \frac{\lct(\fb_{\bullet})}{m}\leq \lct(\cJ(R^G, m\cdot
  \fb_{\bullet}))
  \leq \frac{\lct(\fb_{\bullet})}{m-\lct(\fb_{\bullet})}
 \]
 for any $m>\lct(\fb_{\bullet})$.
 In particular, we have that 
 \[
  \limsup_{m\to\infty} m\cdot \lct(\cJ(R^G, m\cdot
  \fb_{\bullet}))\leq \limsup_{m\to\infty}\frac{m\cdot \lct(\fb_{\bullet})}
  {m-\lct(\fb_{\bullet})} =\lct(\fb_{\bullet}).
 \]
 Since $\fb_s R=(\fa^s)^G R\subset \fa^s$, we know that
 \[
  \lct(\fb_{\bullet})=\lim_{s\to\infty}s\cdot \lct(\fb_s)
  =\lim_{s\to\infty}s\cdot\lct(\fb_s R)\leq \lim_{s\to\infty}s\cdot
  \lct(\fa^s)=\lct(\fa).
 \]
 Combining the last two inequalities, we have
 \begin{equation}\label{ineq2}
   \limsup_{m\to\infty} m\cdot \lct(\cJ(R^G, m\cdot
  \fb_{\bullet}))\leq\lct(\fa).
 \end{equation}
 Combining \eqref{ineq1} and \eqref{ineq2} yields
 \begin{align*}
  \lct(\fa)^n\cdot \mult(\fa)& \geq \limsup_{m\to\infty}
  \lct(\cJ(R^G, m\cdot\fb_{\bullet}))^n\cdot
  \mult(\cJ(R^G, m\cdot\fb_{\bullet})R)\\
  & =|G|\cdot\limsup_{m\to\infty}
  \lct(\cJ(R^G, m\cdot\fb_{\bullet}))^n\cdot
  \mult(\cJ(R^G, m\cdot\fb_{\bullet}))\\
  & \geq |G|\cdot\inf_{\fb\colon\fm^G\textrm{-primary}}
  \lct(\fb)^n\mult(\fb).
 \end{align*}
Hence we prove the lemma. 
\end{proof}

\begin{thm}\label{ghvol}
 Let $(\tilde{X},\tilde{x})$ be a complex klt singularity
 with a faithful action of a finite group $G$ that is 
 free in codimension $1$. Let $(X,x):=(\tilde{X},\tilde{x})/G$.
 Then 
 \begin{enumerate}
 \item We have
 \[
  \hvol^{G}(\tilde{x}, \tilde{X}) =|G|\cdot \hvol(x, X)= \inf_{\substack{\fa\colon\fm_{\tilde{x}}\textrm{-primary}\\
 G\textrm{-invariant}}}
 \lct(\tilde{X};\fa)^n\cdot\mult(\fa).
 \]
 \item For any subgroup $H\subsetneq G$, we have
 $$\hvol^G(\tilde{x},\tilde{X})< [G:H]\cdot\hvol^H(\tilde{x},\tilde{X}).$$ In particular, if $|G|\geq 2$ then $\hvol(x,X)<\hvol(\tilde{x},\tilde{X})$.
 \item If moreover we assume $(\tilde{X},\tilde{x})$ is quasi-regular
 in the sense of \cite{LX16}, that is, the $\hvol(\tilde{X},\tilde{x})$ is calculated by the normalized volume of a divisorial valuation over $x$, then we have
 \[
\hvol(\tilde{x}, \tilde{X})=  \hvol^G(\tilde{x}, \tilde{X}) =|G|\cdot \hvol(x, X)
 \]

 \end{enumerate}
\end{thm}

\begin{proof}
 (1) Firstly, for any $G$-invariant valuation $v$
 on $\tilde{X}$ centered at $\tilde{x}$, we have 
 $$\hvol(v)\geq \lct(\fa_{\bullet}(v))^n \cdot\mult(\fa_{\bullet}(v))$$ 
 by the proof of Theorem \ref{eq-ideal} in \cite{Liu16}.
 Since $\fa_\bullet(v)$ is a graded sequence of $G$-invariant
 ideals, we have
 \begin{equation}\label{ineq3}
 \hvol^G(\tilde{x},\tilde{X})\geq 
  \inf_{\substack{\fa\colon\fm_{\tilde{x}}\textrm{-primary}\\
 G\textrm{-invariant}}}
 \lct(\tilde{X};\fa)^n\cdot\mult(\fa).
 \end{equation}

 Secondly, applying Lemma \ref{finite} 
 and \cite[Theorem 27]{Liu16} yields
 \begin{equation}\label{ineq4}
  \inf_{\substack{\fa\colon\fm_{\tilde{x}}\textrm{-primary}\\
 G\textrm{-invariant}}}
 \lct(\tilde{X};\fa)^n\cdot\mult(\fa)
 \geq |G| \inf_{\fb\colon\fm_{x}\textrm{-primary}}
 \lct(X;\fb)^n\cdot\mult(\fb)
 =|G|\cdot\hvol(x,X).
 \end{equation}
 
 Finally, let $S$ be an arbitrary Koll\'ar component on $(X,x)$.
 By \cite{LX16}, $\pi^* S$ is also a Koll\'ar component up to
 scaling, where $\pi:\tilde{X}\to X$ is the quotient map.
 It is clear that $A_{\tilde{X}}(\pi^* S)=A_X(S)$ and
 $\vol(\pi^* S)=|G|\cdot\vol(S)$. Hence  Theorem \ref{eq-model} implies
 \begin{equation}\label{ineq5}
  \hvol^G(\tilde{x},\tilde{X})
 \leq \inf_{S}\hvol(\pi^*S)=|G|\cdot\inf_S\hvol(S)
 =|G|\cdot\hvol(x,X).
 \end{equation}
 The proof of (1) is finished by combining \eqref{ineq3}, 
 \eqref{ineq4} and \eqref{ineq5}.
 \smallskip
 
 (2) Denote by $(R,\fm):=(\cO_{\tilde{X},\tilde{x}},\fm_{\tilde{x}})$. By (1) it suffices to show that 
 \begin{equation}\label{ineq6}
  \inf_{\substack{\fb\colon\fm\textrm{-primary}\\ G\textrm{-invariant}}}\lct(\fb)^n\cdot\mult(\fb)< [G:H]\cdot\hvol^H(\tilde{x},\tilde{X}).
 \end{equation}
 By \cite{Blu16}, there exist a sequence of $\fm^H$-primary ideals $\fc_m$ of $R^H$ and $\delta>0$, such that
 \[
 \lct(\fc_\bullet)^n\cdot\mult(\fc_\bullet)
 =\inf_{\fc\colon\fm^H\textrm{-primary}}\lct(\fc)^n\cdot\mult(\fc)=\frac{\hvol^H(\tilde{x},\tilde{X})}{|H|}
 \]
 and $\fc_m\subset(\fm^H)^{\lfloor\delta m\rfloor}$ for all $m$. Let us pick $g_1=\id,g_2,\cdots,g_{[G:H]}\in G$ such that $\{g_i H\}=G/H$ is the set of left cosets of $H$ in $G$.
 Let $\fb_m:=\bigcap_{i=1}^{[G:H]}g_i(\fc_m R)$. It is clear that $(\fb_\bullet)$  is a graded sequence of $G$-invariant $\fm$-primary ideals. Since $\fb_m\subset\fc_m R$, we have $\lct(\fb_m)\leq \lct(\fc_m R)=\lct(\fc_m)$ which implies $\lct(\fb_{\bullet})\leq \lct(\fc_\bullet)$.
 Since 
 $$\ell(R/\fa\cap\fb)=\ell(R/\fa)+\ell(R/\fb)-\ell(R/(\fa+\fb)),$$ by induction we have
 \begin{align*}
 \ell(R/\fb_m) & =\ell(R/\bigcap_{i=1}^{[G:H]}g_i(\fc_m R))\\
 & = \ell(R/\fc_m R)+\ell(R/\bigcap_{i=2}^{[G:H]} g_i(\fc_m R)) - \ell(R/(\fc_m R+\bigcap_{i=2}^{[G:H]} g_i(\fc_m R)))\\
 & \leq \sum_{i=1}^{[G:H]} \ell(R/g_i(\fc_m R))-\ell(R/\sum_{i=1}^{[G:H]}g_i(\fc_m R)) \\
 &  = [G:H]\cdot\ell(R/\fc_m R)-\ell(R/\sum_{i=1}^{[G:H]}g_i(\fc_m R)).
 \end{align*}
 Since $\fc_m\subset(\fm^H)^{\lfloor\delta m\rfloor}$, we have $\fc_m R\subset\fm^{\lfloor\delta m\rfloor}$. Thus
 $g_i(\fc_m R)\subset\fm^{\lfloor\delta m\rfloor}$ which implies $\sum_{i=1}^{[G:H]}g_i(\fc_{m}R)\subset\fm^{\lfloor \delta m\rfloor}$.
 Hence we have 
 \begin{align*}
  \mult(\fb_\bullet) &  = \limsup_{m\to\infty} \frac{\ell(R/\fb_m)}{m^n/n!}\\  
  & \leq \limsup_{m\to\infty} \frac{[G:H]\cdot\ell(R/\fc_m R)-\ell(R/\sum_{i=1}^{[G:H]}g_i(\fc_m R))}{m^n/n!}\\
  & \leq [G:H]\cdot\mult(\fc_\bullet  R)-\lim_{m\to\infty}\frac{\ell(R/\fm^{\lfloor \delta m\rfloor})}{m^n/n!}\\
  & = |G|\cdot\mult(\fc_\bullet)-\delta^n\mult(\fm)\\ & <  |G|\cdot\mult(\fc_\bullet).
 \end{align*}
 Therefore,
 \begin{align*}
  \lim_{m\to\infty}\lct(\fb_m)^n\cdot\mult(\fb_m)&
  =\lct(\fb_{\bullet})^n\cdot \mult(\fb_\bullet)\\
  &< |G|\cdot\lct(\fc_\bullet)^n\cdot\mult(\fc_\bullet)\\
  & =[G:H]\cdot\hvol^H(\tilde{x},\tilde{X}).
 \end{align*}
 This finishes the proof of \eqref{ineq6}.
 \smallskip
 
 (3) Since $(\tilde{X},\tilde{x})$ is quasi-regular,
 there exists a unique Koll\'ar component $S$ over $(\tilde{X},\tilde{x})$ such that $\ord_S$ minimizes $\hvol$ in $\Val_{\tilde{X},\tilde{x}}$
 by \cite{LX16}.
 It is clear that $\hvol(g^*\ord_S)=\hvol(\ord_S)$ for any $g\in G$,
 hence $g^*S=S$ since they are both Koll\'ar components minimizing $\hvol$.
 Thus $\ord_S\in\Val_{\tilde{X},\tilde{x}}^G$.
\end{proof}

\begin{rem}\label{r-model}
If a finite group $G$ acts on a singularity $\tilde{x}\in \tilde{X}$, we can also consider $G$-equivariant birational models $Y\to \tilde{X}$ and study the volume of the model $\vol(Y/\tilde{X})$ (see \cite[Definition 3.3]{LX16}). By running an equivariant minimal model program, it is not hard to follow the arguments in \cite{LX16} verbatim to show that
$$\inf_{G-\textrm{model } Y} \vol(Y/\tilde{X})=\hvol^G(\tilde{x}, \tilde{X}). $$
Then Theorem \ref{ghvol}.1 can be also obtained using the comparison of volume of models under a Galois quotient morphism. 
\end{rem}

\subsection{Volumes under birational morphisms}

\begin{lem}\label{hvolup}
 Let $\phi: Y \to X$ be a birational morphism of normal varieties. Then 
 \begin{enumerate}
 \item For any closed point $y\in Y$ and any valuation $v\in\Val_{Y,y}$, we have $\vol_{X,x}(\phi_*v)\leq \vol_{Y,y}(v)$ where $x:=\phi(y)$.
 \item Assume both $X$ and $Y$ have klt singularities. If $K_Y\leq \phi^* K_X$, then  $\hvol(x,X)\leq \hvol(y,Y)$ for any closed point $y\in Y$ where $x:=\phi(y)$.
 \end{enumerate}
\end{lem}

\begin{proof}
(1) Denote by $\phi^\#:\cO_{X,x}\to\cO_{Y,y}$ the injective local ring homomorphism induced by $\phi$. Then it is clear that $\phi^\#(\fa_m(\phi_*v))=\fa_m(v)\cap\phi^\#(\cO_{X,x})$.
 Hence $\cO_{X,x}/\fa_m(\phi_*v)$ injects into 
 $\cO_{Y,y}/\fa_m(v)$ under $\phi^\#$, which implies
 $$\ell(\cO_{X,x}/\fa_m(\phi_*v)) \leq \ell(\cO_{Y,y}/\fa_m(v))$$
 and we are done.
 
(2) Since $K_Y\leq \phi^*K_X$, we have $A_X(\phi_*v)\leq A_Y(v)$ for any (divisorial) valuation $v\in\Val_{Y,y}$. Hence (2) follows from (1).
\end{proof}

\bigskip

In the below, we aim to prove a stronger result Corollary \ref{c-strict} on the comparison of volumes under a birational morphism. 

\begin{lem}\label{asymcoh}
 Let $Y$ be a normal projective variety. Let $L$ be a nef and big line bundle on $Y$. Let $y\in C$ be a closed point where $C$ is a curve satisfying  $(C\cdot L)=0$.  Then  there exists $\epsilon>0$ such that  
 $$h^1(Y,L^{\otimes k}\otimes\fm_y^k)\geq \epsilon k^n \mbox{ for } k\gg 1.$$
\end{lem}

\begin{proof}
Let $\psi:\widehat{Y}\to Y$ be the normalized blow up of $y$, i.e. $\widehat{Y}$ is the normalization of ${\rm Bl}_yY$. Thus we can write the relative anti-ample Cartier divisor $\mathcal{O}(-1)$ as 
$$E=\pi^{-1}(\fm_y)=\sum a_i E_i,$$ where $E_i$ are the prime components and $a_i$ is the coefficient for $E_i$. 
By \cite[Corollary C]{LM09}, we know that 
 \begin{equation}\label{asym1}
\vol_{\widehat{Y}}(\psi^*L)-\vol_{\widehat{Y}}(\psi^*L-E)=n\int_{0}^1
\vol_{\widehat{Y}|E}(\psi^*L-tE)dt,
 \end{equation}
 where $\vol_{\widehat{Y}|E}(\psi^*L-tE):=\sum a_i\vol_{\widehat{Y}|E_i}(\psi^*L-tE)$.
 Let $\widehat{C}$ be the birational transform of $C$ in $\widehat{Y}$. Then for a fixed $t\in(0,1]\cap\bQ$ we know that $$((\psi^*L-tE)\cdot\widehat{C})=(L\cdot C)-t(E\cdot\widehat{C})<0.$$ Hence by \cite[Proposition 1.1]{dFKL07} there exist a positive integer $q=q(t)$ such that 
 $$\fb(|k(\psi^*L-tE)|)\subset I_{\widehat{C}}^{\lfloor k/q\rfloor} \mbox{ for }k\gg 1 \mbox{ with }kt\in\bZ,$$
 where $\fb(|\cdot|)$ means the base ideal of a linear system. Pick a closed point $\hat{y}\in\widehat{C}\cap \Supp(E)$, then $\fb(|k(\psi^*L-tE)|)\subset\fm_{\hat{y}}^{\lfloor k/q\rfloor}$. Let $E_i$ be a component of $E$ containing $\hat{y}$. Then for each divisor $D\in |k(\psi^*L-tE)|$ that does not contain $E_i$ as a component, we have $\ord_{\hat{y}}(D|_{E_i})\geq \lfloor k/q\rfloor$. Since $\psi^*L|_{E_i}$ is a trivial line bundle and $(-E)|_{E_i}\sim\cO_{E_i}(1)$ by the definition of $E$, we have an inclusion
 \[
 \mathrm{image}(H^0(\widehat{Y},k(\psi^*L-tE))\to H^0(E_i,\cO_{E_i}(kt)))\subset H^0(E_i, \cO_{E_i}(kt)\otimes\fm_{\hat{y}}^{\lfloor k/q\rfloor}).
 \]
 Thus we have 
 \[
 \vol_{\widehat{Y}|E_i}(\psi^*L-tE)\leq \limsup_{k\to\infty}\frac{h^0(E_i, \cO_{E_i}(kt)\otimes\fm_{\hat{y}}^{\lfloor k/q\rfloor})}{k^{n-1}/(n-1)!}.
 \]
 Let us pick $q'\geq q$ such that $q'\cdot\epsilon(\cO_{E_i}(1),\hat{y})> t$ where $\epsilon(\cO_{E_i}(1),\hat{y})$ is the Seshadri constant. Then the following sequence is exact for $k\gg 1$ with $kt\in\bZ$:
\[
0\to H^0(E_i, \cO_{E_i}(kt)\otimes\fm_{\hat{y}}^{\lfloor k/q'\rfloor})\to H^0(E_i, \cO_{E_i}(kt))\to H^0(E_i,(\cO_{E_i}/\fm_{\hat{y}}^{\lfloor k/q'\rfloor})(kt))\to 0,
\]
as $H^1(E_i, \cO_{E_i}(kt)\otimes\fm_{\hat{y}}^{\lfloor k/q'\rfloor})=0$. 
Hence we have
 \begin{align*}
 \vol_{\widehat{Y}|E_i}(\psi^*L-tE)&\leq \limsup_{k\to\infty}\frac{h^0(E_i, \cO_{E_i}(kt)\otimes\fm_{\hat{y}}^{\lfloor k/q'\rfloor})}{k^{n-1}/(n-1)!}\\
 & =\limsup_{k\to\infty}\frac{h^0(E_i, \cO_{E_i}(kt))-\ell(\cO_{\hat{y},E_i}/\fm_{\hat{y}}^{\lfloor k/q'\rfloor})}{k^{n-1}/(n-1)!}\\
 & =\vol_{E_i}((-tE)|_{E_i})-\frac{\mult_{\hat{y}}E_i}{(q')^{n-1}}\\ & <\vol_{E_i}((-tE)|_{E_i}).
 \end{align*}
For any $j\neq i$, we have $\vol_{\widehat{Y}|E_j}(\psi^*L-tE)\leq \vol_{E_j}((-tE)|_{E_j})$. Hence we have
\begin{equation}\label{asym2}
\vol_{\widehat{Y}|E}(\psi^*L-tE)<\sum a_i\vol_{E_i}((-tE)|_{E_i})
 =-(-E)^n t^{n-1}=\mult_y Y\cdot t^{n-1}.
\end{equation}
Notice that the inequality above works for all $t\in(0,1]\cap\bQ$. Since $\vol_{\widehat{Y}|E}(\psi^*L-tE)$ is continuous in $t$, \eqref{asym1} and \eqref{asym2} implies that
\[
\vol_{\widehat{Y}}(\psi^*L-E)>\vol_{\widehat{Y}}(\psi^*L)-n\int_0^1\mult_y Y\cdot t^{n-1} = \vol_{Y}(L) - \mult_y Y.
\]
Thus we have for $k\gg 1,$
\begin{align*}
h^1(Y,L^{\otimes k}\otimes\fm_y^k)&\ge h^1(Y,L^{\otimes k}\otimes \psi_*(\mathcal{O}_{\widehat{Y}}(-kE)) )\\
 & = h^1(\widehat{Y},\psi^*L^{\otimes k}\otimes \mathcal{O}_{\widehat{Y}}(-kE)) \\
 & =(\mult_y Y-\vol_{Y}(L)+\vol_{\widehat{Y}}(\psi^*L-E))\frac{k^n}{n!}+O(k^{n-1})\\
 &\ge \epsilon k^n.
 \end{align*}
Hence we finish the proof.
\end{proof}

\begin{lem}\label{hvolstrict}
 Let $\phi:(Y,y)\to (X,x)$ be a birational morphism of normal varieties with $y\in\Ex(\phi)$. Then for any valuation $v\in\Val_{Y,y}$ satisfying Izumi's inequality, that is, there exists $c_1\ge c_2>0$ such that $\frac{1}{c_2}\cdot {\ord_y}\ge v \ge \frac{1}{c_1}\cdot {\ord_y}$, we have $\vol_{X,x}(\phi_*v)<\vol_{Y,y}(v)$.
\end{lem}

\begin{proof}
 Let us take suitable projective closures of $X,Y$ such that $\phi$ extends to a birational morphism between normal projective varieties. Denote by $\fa_k:=\fa_k(v)$ and $\fb_k:=\fa_k(\phi_*v)$. Then by \cite[Lemma 3.9]{LM09} there exists an ample line bundle $M$ on $X$ such that for every $k,i>0$ we have $H^i(X,M^{\otimes k}\otimes \fb_k)=0$. Thus we have 
 \begin{equation}\label{asym3}
 \limsup_{k\to \infty}\frac{h^0(X,M^{\otimes k}\otimes \fb_k)}{k^n/n!}
 =\vol_{X}(M)-\vol_{X,x}(\phi_*v).
 \end{equation}
Since $v$ satisfies the Izumi's inequality, we know $\fm_y^{\lceil c_1 k\rceil}\subset\fa_k\subset\fm_y^{\lfloor c_2 k\rfloor}$ for $k\gg 1$ for the choice of $c_1, c_2$. Thus we have the following relations
\begin{align*}
  &  H^1(Y,(\phi^*M)^{\otimes k}\otimes \fa_k)\twoheadrightarrow H^1(Y,(\phi^*M)^{\otimes k}\otimes \fm_y^{\lfloor c_2 k\rfloor}),\\
  & H^i(Y,(\phi^*M)^{\otimes k}\otimes \fa_k)\xrightarrow[]{\cong} H^i(Y,(\phi^*M)^{\otimes k})\quad\textrm{ for any } i\geq 2.
\end{align*}
Since $h^i(Y,(\phi^*M)^{\otimes k})=O(k^{n-1})$, we have
\begin{align*}
\limsup_{k\to \infty}\frac{h^0(Y,(\phi^*M)^{\otimes k}\otimes \fa_k)}{k^n/n!}
 &=\vol_{Y}(\phi^*M)-\vol_{Y,y}(v)+\limsup_{k\to \infty}\frac{h^1(Y,(\phi^*M)^{\otimes k}\otimes \fa_k)}{k^n/n!}\\
 &\geq \vol_{X}(M)-\vol_{Y,y}(v)+\limsup_{k\to \infty}\frac{h^1(Y,(\phi^*M)^{\otimes k}\otimes \fm_y^{\lfloor c_2 k\rfloor})}{k^n/n!}
\end{align*}
By Lemma \ref{asymcoh}, there exists $\epsilon>0$ such that
$h^1(Y,(\phi^*M)^{\otimes k}\otimes \fm_y^{\lfloor c_2 k\rfloor})\geq \epsilon k^n$ for $k\gg 1$. Thus
\begin{equation}\label{asym4}
 \limsup_{k\to \infty}\frac{h^0(Y,(\phi^*M)^{\otimes k}\otimes \fa_k)}{k^n/n!}
 \geq \vol_{X}(M)-\vol_{Y,y}(v)+ n!\epsilon.
\end{equation}
Since $\fb_k=\phi_*\fa_k$, we know that the left hand sides of \eqref{asym3} and \eqref{asym4} are the same. As a result,
\[
\vol_X(M)-\vol_{X,x}(\phi_*v)\geq \vol_{X}(M)-\vol_{Y,y}(v)+ n!\epsilon
\]
and we finish the proof.
\end{proof}

\begin{cor}\label{c-strict}
 Let $\phi:(Y,y)\to (X,x)$ be a birational morphism of klt singularities such that $y\in\Ex(\phi)$.  If $K_Y\leq \phi^* K_X$, then $\hvol(x,X)<\hvol(y,Y)$.
\end{cor}

\begin{proof}
 Let $v_*\in\Val_{Y,y}$ be a minimizer of $\hvol$ whose existence was proved in \cite{Blu16}. Then $A_Y(v_*)<+\infty$ which implies that $v_*$ satisfies Izumi's inequality by \cite[Proposition 2.3]{Li15a}. By Lemma \ref{hvolstrict} we know that $\vol_{X,x}(\phi_* v_*)<\vol_{Y,y}(v_*)$. The assumption $K_Y\leq \phi^*K_X$ implies that $A_{X}(\phi_* v_*)\leq A_{Y}(v_*)$. Hence the proof is finished.
\end{proof}

\section{ Volumes of threefold singularities}

This section is the main technical contribution of our paper. We will estimate some upper bound of the volume of three dimensional klt singularities and prove Theorem \ref{t-local}. 
We first prove that any three dimensional singular point will have volume at most 16, which is Theorem \ref{t-local}.1.
We obtain this by going through some  explicit descriptions.
Then taking the action by a cyclic group  into the account
of the analysis, we obtain Theorem \ref{t-local}.3-4. 
A similar argument will also give the proof of Theorem \ref{fundgp}. 
Note that Theorem \ref{t-local}.2 is a direct consequence of \cite[Example 7.1.1]{LX16}.


\subsection{Estimate on the local volume}

\begin{lem}\label{hypersurf}
 Let $(X,x)$ be a canonical hypersurface singularity of dimension $n$. Assume a finite group $G$ acts on $(X,x)$. Then
 \begin{equation}\label{eq_hypersurf}
 \hvol^G(x,X) \leq (n+1-\mult_x X)^n\cdot\mult_x X \leq n^n.
 \end{equation}
 In particular, $\hvol^G(x,X)\leq 2(n-1)^n$ unless $(X,x)$ is smooth.
\end{lem}

\begin{proof}
  Fix an embedding $(X,x)\subset (\bA^{n+1},o)$, consider the blow up $\phi: \Bl_{o}\bA^{n+1}\to \bA^{n+1}$ with exceptional divisor $E$. Let $Y:=\phi_*^{-1}(X)$, then by adjunction we have that 
  $$\omega_Y\cong \phi^*\omega_X\otimes\cO_Y((n-\mult_x X)E|_{Y}).$$ Let $\nu:\overline{Y}\to Y$ be the normalization, then $\omega_{\overline{Y}}\cong \cI\cdot\nu^*\omega_Y$ where $\cI\subset\cO_{\overline{Y}}$ is the conductor ideal. As a result, we have  
  \[
  K_{\overline{Y}}\leq (\phi\circ\nu)^* K_X+(n-\mult_x X)\nu^*(E|_{Y}).
  \]
  Let $F$ be a prime divisor in $\nu^*(E|_Y)$ of coefficient $a\geq 1$. Then we have $A_X(\ord_F)\leq 1+(n-\mult_x X)a$ and $\ord_F(\fm_x)=\ord_F(\nu^*(E|_Y))=a$. Thus we have
  \[
  \lct(X;\fm_x)\leq \frac{A_X(\ord_F)}{\ord_F(\fm_x)}\leq\frac{1+(n-\mult_x X)a}{a}\leq n+1-\mult_x X.
  \]
  It is clear that the maximal ideal $\fm_x$ is $G$-invariant. 
  Hence \eqref{eq_hypersurf} follows from $\lct(X;\fm_x)\leq n+1-\mult_x X$ by Theorem \ref{ghvol}.
\end{proof}

The following structure result for strictly canonical equivariant singularities is useful to reduce the estimate of their volumes to the case of terminal singularities. 

\begin{lem}\label{maxcrep}
 Let $X$ be a variety with canonical singularities. Assume a finite group $G$ acts on $X$. Then there exists a $G$-equivariant proper birational morphism $\phi:Y\to X$ such that $Y$ is $G$-$\bQ$-factorial terminal and $K_Y=\phi^* K_X$. If moreover that $X$ is Gorenstein, then such $Y$ is Gorenstein as well. Such $Y$ will be called a \emph{$G$-equivariant maximal crepant model} over $X$.
\end{lem}

\begin{proof}
 By \cite[Proposition 2.36]{KM98}, there are only finitely many crepant exceptional divisors over $X$. Then the lemma is an easy consequence of \cite{BCHM10} (see also \cite[Corollary 1.38]{Kol13}) by extracting all crepant exceptial divisors over $X$ using $G$-MMP.
\end{proof}

We need a lemma on del Pezzo surfaces.
\begin{lem}\label{l-normal}
If $E$ is a normal Gorenstein surface such that $-K_E$ is ample, then either $E$ is a generalized cone over an elliptic curve or $E$ only has rational double points. 
\end{lem}
\begin{proof}See \cite[Theorem 1]{Bre80}.
\end{proof}

Next we will prove the inequality part of Theorem \ref{t-local}.1.
Let $x\in X$ be a three dimensional non-smooth klt singularity.
Denote the index $1$ cover of $x\in X$ by $\tilde{x}\in\tilde{X}$.
If $\tilde{x}\in \tilde{X}$ is smooth, then $x\in X$ is a quotient singularity
which implies $\hvol(x,X)\leq \frac{27}{2}<16$.
Thus we may assume that $\tilde{x}\in \tilde{X}$ is singular.
Then Theorem \ref{ghvol}.2 implies that $\hvol(x,X)\leq \hvol(\tilde{x},\tilde{X})$.
Hence it suffices to prove the inequality part of Theorem \ref{t-local}.1
under the assumption that $x\in X$ is Gorenstein.
This is shown in the the following result.

\begin{prop}\label{l-weak}
 Let $(X,x)$ be a Gorenstein canonical singularity of dimension $3$. Then $\hvol(x,X)\leq 16$ unless $(X,x)$ is smooth.
\end{prop}

\begin{proof}
Let us first recall a Gorenstein terminal three dimensional singularity is a compound Du Val (or cDV) singularity (see \cite[Definition 5.32 and Theorem 5.34]{KM98}). In particular, it is a hypersurface singularity.  

 Thus we may assume that $(X,x)$ is not a cDV singularity, otherwise we could conclude by Lemma \ref{hypersurf}. Let $\phi_1:Y_1\to X$ be a maximal crepant model constructed in Lemma \ref{maxcrep}. From \cite[Theorem 5.35]{KM98} we know that there exists a crepant $\phi_1$-exceptional divisor $E\subset Y_1$ over $x$. Let us run $(Y_1,\epsilon E)$-MMP over $X$ for $0<\epsilon\ll 1$. By \cite[1.35]{Kol13} this MMP will terminate as $Y_1\dashrightarrow Y\to Y'$, where $Y_1\dashrightarrow Y$ is the composition of a sequence of flips, and $g:Y\to Y'$ contracts the birational transform of $E$, which we also denote by $E$ as abuse of notation. It is clear that $Y$ is a maximal crepant model over $X$ as well, so it is Gorenstein terminal and $\bQ$-factorial. If $\dim g(E)=1$, then the generic point of $g(E)$ is a codimension two point on $Y'$ with a crepant resolution, thus $Y'$ has non-isolated cDV singularities along general points in $g(E)$ (see e.g. \cite[Theorem 4.20]{KM98}). Hence for a general point $y'\in g(E)$, Lemma \ref{hypersurf} and \ref{hvolup} imply that 
 $$\hvol(x,X)\leq\hvol(y',Y')\leq 16$$ and we are done.
 
 The only case left is when $g(E)=y'$ is a point. 
 If $Y$ has a singularity $y$ along $E$, then since it is terminal Gorenstein, we know 
 $$\hvol(y', Y')\le \hvol(y, Y)\le 16.$$
 So we can assume $Y$ is smooth along $E$. In particular, the divisor $E$ is Cartier in $Y$, so $E$ is Gorenstein. Since $(-E)$ is $g$-ample, $K_E=(K_Y+E)|_E=E|_E$ and $A_{Y'}(\ord_E)=1$, it is clear that $E$ is a reduced irreducible Gorenstein del Pezzo surface, and 
 $$\hvol(y',Y')\leq \hvol_{y',Y'}(\ord_E)=(E^3)=(-K_E)^2=:m.$$ 
 
 
 Firstly, we treat the case when $E$ is normal. By Lemma \ref{l-normal}, $E$ is either a cone over an elliptic curve or $E$ only has rational double point. 
 In the first case, let $y$ be the cone point, then 
 $$m=(-K_E)^2=\edim(E,y)\leq\edim(Y,y)= 3,$$
 where $\edim(E,y)$ and $\edim(Y,y)$ denote the embedding dimension of the singularities. 
 In the latter case, 
 $$m=(-K_{E})^2\leq 9.$$

Now the only case left is when $E$ is non-normal. Since $E$ is a  reduced irreducible non-normal Gorenstein del Pezzo surfaces, Reid's classification \cite[Theorem 1.1]{Rei94} tells us that $E$ is one of the following:
\begin{itemize}
\item The degree $m=(-K_E)^2$ is $1$ or $2$, such $E$ is classified by \cite[1.4]{Rei94};
\item Cone over a  nodal/cuspidal rational curve, then the assumption that $\edim(E,y)\le 3$ implies $m\le 3$;
\item A linear projection of Veronese $\bP^2\subset\bP^5$, then $m=(-K_E)^2=4$;
\item A linear projection of $F_{m-2;1}$ by identifying
a fiber with the negative section;
\item A linear projection of $F_{m-4;2}$ by identifying 
the negative section to itself via an involution.
\end{itemize}

We recall the definition of $F_{a;k}$ in \cite{Rei94}.
For an integer $a\geq 0$, denote the $a$-th Hirzebruch 
surface by $F_a$ with $A$ and $B$ being its fiber and negative
section, respectively. Then for any integer $k>0$,
the complete linear system $|(a+k)A+B|$ induces
an embedding of $F_a$ into $\bP^{a+2k+1}$. 
We denote the image by $F_{a;k}$. It is clear that 
$F_{a;k}$ is a surface of degree $a+2k$ in $\bP^{a+2k+1}$,
with negative section $B$ of degree $k$.

In the first three cases, we have $m\leq 4$ and we are done.

 To proceed, we will show that  in the last two cases indeed 
 $m=7$ in both cases.
 Hence $\hvol(y',Y')\leq \hvol_{y',Y'}(\ord_E)=m=7$.
 
 In the fourth case, $E$ is a linear projection of $F_{m-2;1}$, so it is obtained by gluing a fiber
 $A$ with the negative curve $B$ from its normalization
 $\overline{E}\cong F_{m-2}$. Denote by $\tau\colon  A\to B$ the gluing morphism and $\nu:\overline{E}\to E$
 the normalization map. Let $C$ be the reduced curve in $E$ whose
 support is $\nu(A)=\nu(B)$.  By \cite[2.1]{Rei94}
 we have $\cO_E=\ker(\nu_*\cO_{\overline{E}}\to(\nu|_{A\cup B})_*\cO_{A\cup B}/\cO_C)$.
 Here $A\cup B$ is a reduced closed subscheme of $E$.
 Let us pick an analytic local coordinate
 $(u_1,u_2)$ on $\overline{E}$ with origin at $\bar{y}:=A\cap B$
 such that $A=(u_1=0)$ and $B=(u_2=0)$. Denote by $y:=\nu(\bar{y})$.
 After taking completions, we get
 \[
  \widehat{\cO_{\overline{E},\bar{y}}}=\bC[[u_1,u_2]],
  \quad\widehat{\cO_{A\cup B,\bar{y}}}=\bC[[u_1,u_2]]/(u_1 u_2),
  \quad\widehat{\cO_{C,y}}=\bC[[u_1+u_2, u_1u_2]]/(u_1 u_2).
 \]
 Hence $\widehat{\cO_{E,y}}=\bC[[u_1+u_2,u_1 u_2, u_1^2 u_2]]$.
 Denote by $(x_1,x_2,x_3):=(u_1+u_2,u_1u_2, u_1^2u_2)$, then 
 \[
  \widehat{\cO_{E,y}}\cong\bC[[x_1,x_2,x_3]]/(x_3^2+x_2^3-x_1x_2x_3).
 \]
 Since $Y$ is smooth along $E$, we may assume that $(x_1,x_2,x_3)$
 is a local analytic coordinate of $Y$ at $y$.
 Consider the following map:
 \[
  \Phi: (\cT_{\overline{E}}\otimes\cO_{A})\oplus \tau^*(\cT_{\overline{E}}\otimes\cO_B)
  \to (\nu|_{A})^*(\cT_Y\otimes\cO_C).
 \]
 The map $\Phi$ is induced by the tangent map of $\nu:\overline{E}\to Y$.
 Since $\cT_{\overline{E}}=\langle\partial_{u_1},\partial_{u_2}\rangle$,
 $(x_1,x_2,x_3)=(u_1+u_2,u_1u_2, u_1^2u_2)$ and $C=(x_2=x_3=0)$, 
 we have
 \[
  \Phi(\cT_{\overline{E}}\otimes\cO_{A})=\langle\partial_{x_1},\partial_{x_1}+x_1\partial_{x_2}\rangle,
  \quad\Phi(\tau^*(\cT_{\overline{E}}\otimes\cO_{B}))=\langle\partial_{x_1},\partial_{x_1}+x_1\partial_{x_2}+x_1^2\partial_{x_3}\rangle. 
 \]
 Thus $\mathrm{Im}(\Phi)=\langle\partial_{x_1},x_1\partial_{x_2},x_1^2\partial_{x_3}\rangle$.
 This implies that the cokernel of $\Phi$ near $\bar{y}$ is a
 skyscraper sheaf $\cG$ supported at $\bar{y}$ with stalk isomorphic to $\bC^3$.
 Therefore, we have an exact sequence
 \[
  0\to \cT_C\to(\cT_{\overline{E}}\otimes\cO_{A})\oplus(\tau|_A)^*(\cT_{\overline{E}}\otimes\cO_B)
  \xrightarrow{\Phi}(\nu|_{A})^*(\cT_Y\otimes\cO_C)\to \cG\to 0.
 \]
 Taking degrees of the above exact sequence, we get
 \[
  \deg\cT_C+\deg\cT_Y\otimes\cO_C=\deg\cT_{\overline{E}}\otimes\cO_{A}
  +\deg\cT_{\overline{E}}\otimes\cO_{B}+3
 \]
 which implies $\deg\cN_{C/Y}=\deg\cN_{A/\overline{E}}+\deg\cN_{B/\overline{E}}+3$.
 Since $Y$ is crepant over $Y'$, we know $(K_Y\cdot C)=0$
 which implies $\deg\cN_{C/Y}=-2$. We also know that 
 $\deg\cN_{A/\overline{E}}=0$ and $\deg\cN_{B/\overline{E}}=-(m-2)$.
 Hence $-2=-(m-2)+3$ which means $m=7$.
 
 In the fifth case, $E$ is obtained by gluing the negative section
 via a non-trivial involution $\tau: B\to B$ from $\overline{E}\cong F_{m-4}$.
 Denote by $C:=\nu(B)$. Let $\bar{y}$  be a $\tau$-fixed point in $B$. 
 By exactly the same computation
 in the proof of Lemma \ref{l-equivirr}, we can choose
 a local analytic coordinate $(u_1,u_2)$ of $\overline{E}$
 at $\bar{y}$ such that $B=(u_1=0)$ and $\tau:B\to B$ is given
 by $\tau^*(u_2)=-u_2$. We may also find a local analytic
 coordinate $(x_1,x_2,x_3)$ of $Y$ at $y:=\nu(\bar{y})$
 such that $\nu:\overline{E}\to E\hookrightarrow Y$ 
 has the local expression $\nu(u_1,u_2)=(u_2^2,u_1,u_1u_2)$.
 Consider the following map:
 \[
  \Phi:(\cT_{\overline{E}}\otimes\cO_B)\oplus\tau^*(\cT_{\overline{E}}\otimes\cO_B)
  \to(\nu|_{B})^*(\cT_{Y}\otimes\cO_C).
 \]
 The map $\Phi$ is induced by the tangent map of $\nu:\overline{E}\to Y$.
 Since $\cT_{\overline{E}}=\langle\partial_{u_1},\partial_{u_2}\rangle$,
 $(x_1,x_2,x_3)=(u_2^2,u_1,u_1 u_2)$ and $C=(x_2=x_3=0)$,
 we have 
 \[
  \Phi(\cT_{\overline{E}}\otimes\cO_B)=\langle\partial_{x_2}+u_2\partial_{x_3},
  2u_2\partial_{x_1}\rangle,\quad
  \Phi(\tau^*(\cT_{\overline{E}}\otimes\cO_B))=\langle\partial_{x_2}-u_2\partial_{x_3},
  -2u_2\partial_{x_1}\rangle.
 \]
 Thus $\mathrm{Im}(\Phi)=\langle u_2\partial_{x_1}, \partial_{x_2}, u_2\partial_{x_3}\rangle$.
 Denote by $\cG$ the cokernel of $\Phi$. Then 
 $\cG$ restricting to a neighborhood of $\bar{y}$
 is a skyscraper sheaf $\cG$ with stalk isomorphic to $\bC^2$.
 Since there are two $\tau$-fixed points in $B$,
 the total length of $\cG$ on $B$ is $4$. Thus we obtain
 the following exact sequence:
 \[
  0\to\cT_B\to(\cT_{\overline{E}}\otimes\cO_B)\oplus\tau^*(\cT_{\overline{E}}\otimes\cO_B)
  \xrightarrow{\Phi}(\nu|_{B})^*(\cT_{Y}\otimes\cO_C)\to\cG\to 0.
 \]
 Taking degrees of the above exact sequence, we get
 \[
  \deg\cT_B+\deg\cT_Y\otimes\cO_C=2\deg\cT_{\overline{E}}\otimes\cO_{B}
  +4
 \]
 which implies $\deg\cN_{C/Y}=2\deg\cN_{B/\overline{E}}+4$.
 As before,  $\deg\cN_{C/Y}=-2$ and
 $\deg\cN_{B/\overline{E}}=-(m-4)$, hence $-2=-2(m-4)+4$ which means $m=7$.
\end{proof}
\medskip

\noindent {\bf Complement of the proof:} In the last two cases, we will provide the second type of valuations which have normalized volume at most 16. A similar study will play an essential role in the latter argument of Theorem \ref{t-local}.4  for the equivariant case, because there the above argument seems to be not enough to produce an {\it equivariant} valuation with volume at most 16, therefore we need to find a new candidate  (see the proof of Lemma \ref{l-equivirr}).

Let $\nu:\overline{E}\to E$ be the normalization map.
It is clear that $E$ is nodal in codimension $1$, hence $\nu$ is unramified away from finitely many points on $\overline{E}$. Let $\bar{l}$ be a general fiber of the Hirzebruch surface $\overline{E}$. Denote by $l:=\nu_*\bar{l}$. 

We will show that $\hvol_{Y',y'}(\ord_l)\leq 16$. Since we can choose $\bar{l}$ sufficiently general to miss the finite points where $\nu$ is ramified, we may assume that $\nu$ is unramified along $\bar{l}$, which implies that $\Omega_{\overline{E}/E}\otimes\cO_{\bar{l}}=0$.
From the following exact sequence
\[
(\nu^*\Omega_{E})\otimes\cO_{\bar{l}}\to\Omega_{\overline{E}}\otimes\cO_{\bar{l}}\to \Omega_{\overline{E}/E}\otimes\cO_{\bar{l}}=0,
\]
we have a sequence of surjections after applying $\nu_*$:
\begin{equation}\label{normbdl}
 \Omega_Y\otimes\cO_{l}\twoheadrightarrow\Omega_{E}\otimes\cO_{l}\twoheadrightarrow\nu_*(\Omega_{\overline{E}}\otimes\cO_{\bar{l}})\twoheadrightarrow\nu_*\Omega_{\bar{l}}=\Omega_{l}.
\end{equation}
Since  $Y$ is smooth along $l$, the conormal sheaf $\cN_{l/Y}^{\vee}$ is a vector bundle. 
It is clear that 
$$\ker(\Omega_{\overline{E}}\otimes\cO_{\bar{l}}\to\nu_*\Omega_{\bar{l}})=\cN_{\bar{l}/\overline{E}}^{\vee}\cong\cO_{\bar{l}},$$ hence the sequence \eqref{normbdl} gives us a surjection $\cN_{l/Y}^{\vee}\twoheadrightarrow\cO_{l}$. Since $\omega_{Y}\otimes\cO_{l}\cong\cO_{l}$, we know that $\deg\cN_{l/Y}^{\vee}=2$. Hence we have an exact sequence
\[
 0\to\cO_{l}(2)\to\cN_{l/Y}^{\vee}\to\cO_{l}\to 0
\]
which splits because $\Ext^1(\cO_{l},\cO_l(2))=0$. Hence $\cN_{l/Y}^{\vee}\cong\cO_l\oplus\cO_l(2)$. Let $\cI_l$ be the ideal sheaf of $l$ in $Y$. Then we have an injection $(\fa_k/\fa_{k+1})(\ord_{l})\hookrightarrow H^0(l,\cI_l^k/\cI_l^{k+1})$ where $\fa_k(\ord_{l})$ is the $k$-th valuative ideal of $\ord_l$ in $\cO_{Y',y'}$. Since 
\[
\cI_l^k/\cI_l^{k+1}\cong\Sym^k\cN_{l/Y}^{\vee}\cong\cO_l\oplus\cO_l(2)\oplus\cdots\oplus\cO_l(2k),
\]
we know that $h^0(l,\cI_l^k/\cI_l^{k+1})=k^2+O(k)$. Hence $\ell(\cO_{Y',y'}/\fa_k(\ord_l))\leq k^3/3+O(k^2)$ which implies that $\vol_{Y',y'}(\ord_l)\leq 2$. Since $A_{Y'}(\ord_l)=2$, we have $$\hvol(y',Y')\leq\hvol_{Y',y'}(\ord_l)=8\cdot \vol_{Y',y'}(\ord_l)\leq 16$$ and the proof is finished.

 

\bigskip

Next,  we treat the equality case of Theorem \ref{t-local}.1.
\begin{prop}A three dimensional klt singularity $x\in X$ has $\hvol(x,X)=16$ if and only if $x\in X$ is an $A_1$ singularity.
\end{prop}
\begin{proof}Let $\tilde{x}\in \tilde{X}$ be the index $1$ cover of $x\in X$, then by Proposition \ref{l-weak} if $\hvol(\tilde{x},\tilde{X})>16$, then it is a smooth point, so 
$\hvol(x,X)=27/|G|$ for some nontrivial $G$ and we get a contradiction.
So by
$$16\ge \hvol(\tilde{x},\tilde{X}) $$
and Theorem \ref{ghvol}, we indeed know $(\tilde{x},\tilde{X})=(x,X)$. By the proof of Proposition \ref{l-weak}, we see that $x\in X$ is either a cDV point or it has a crepant resolution which extracts a non-normal del Pezzo surface $E$ and a blow up of a curve $l\subset E$ yields a valuation with normalized volume  at most 16. 

Firstly let us consider that $x\in X$ is a cDV point. If $x\in X$ is not of cA type, then it is locally given by an equation $(x_1^2+f(x_2,x_3,x_4)=0)$ where $\ord_0 f\geq 3$. Then the model $Y'\to X$ given by the normalized weighted blow up of $(3,2,2,2)$ satisfies $\vol(Y'/X)\leq\frac{27}{4}$. In fact, by definition $Y'$ is the normalization of a variety $Y^*\subset W:={\rm Bl}_{(3,2,2,2)}{\mathbb{C}^4}$. The exceptional set $Y^*/X$ is a degree 6 hypersurface $F:=Y^*\cap \mathbb{P}(3,2,2,2)$. An application of the adjunction formula says
$(K_{Y^*}+F)\sim_{\mathbb{Q},X} 3F.$
Let $H$ denote the class of $\mathcal{O}(1)$ on $\mathbb{P}(3,2,2,2)$. As the normalization will only possibly decrease the volume, we know that
$$\vol(Y'/X)\le 3^3\cdot (H|_{F})^3=\frac{27\times 6}{3\cdot 2 \cdot 2 \cdot 2}=\frac{27}{4}<16.$$

Hence we may assume that $x\in X$ is of $cA$-type. Let $Y\to X$ be its blow up of $x$ with exceptional divisor $E$. Since $x\in X$ is of $cA$-type, we know that $E$ is reduced which implies $Y$ is normal by Serre's criterion. If $Y$ is singular along $E$ at some point $y$, then by Corollary \ref{c-strict}, 
$$\hvol(x,X)<\hvol(y,Y)\le 16.$$ Therefore, $Y$ is smooth. Furthermore, if the model $Y\to X$ does not give a Koll\'ar component, then by \cite[Theorem C.2 and its proof]{LX16}, we know there is a Koll\'ar component $S$, such that 
$$\hvol_{X,x}(\ord_S)<\vol(Y/X)=16.$$ Thus $Y\to X$ has to be a Koll\'ar component. This implies that $x\in X$ is of $A_k$ type since otherwise the exceptional divisor of $Y/X$ is given by $(x_1^2+x_2^2=0)\subset \mathbb{P}^3$, which is clearly a contradiction. In fact, as $Y$ is smooth, it is either $A_1$ or $A_2$. For $A_2$ singularity, it is known that its volume is $\frac{125}{9}$ computed on the valuation $v$ induced by the weighted blow up of $(3,3,3,2)$. In fact, the three dimensional $A_2$ singularity is labelled as the case $(n,k)=(3,3)$ in \cite[Example 5.3]{Li15a}, where $v$ is  proposed to be minimizer and $\hvol(v)$ is explicitly calculated. Then \cite[Example 4.7]{LL19} confirms that  $v$ is indeed the minimizer in this case.

Now consider the case that $Y\to X$ extracts a non-normal del Pezzo surface $E$. We want to show that $\hvol(x,X)<16$. From the proof of Proposition \ref{l-weak} we know that $\hvol_{X,x}(\ord_l)\leq 16$ for a general line $l\subset E$. Since a minimizing divisorial valuation of $\hvol$ is unique if exists by \cite[Theorem B]{LX16}, we conclude that either $\hvol_{X,x}(\ord_l)<16$ or $\hvol_{X,x}(\ord_l)=16$ for any general line $l$ but then these valuations $\ord_l$ can not be a minimizer of $\hvol_{X,x}$. Therefore, $\hvol(x,X)<16$ and we finish the proof.
\end{proof}


\subsection{Equivariant estimate}

We have proved Theorem \ref{t-local}.1. Theorem \ref{t-local}.2 follows from \cite[7.1.1]{LX16}. In this section, we aim to prove Theorem \ref{t-local}.3-4. 
We need the following result from \cite{HX09}.
\begin{prop}\label{p-HX}Let $G$ be a finite group and $({x}\in {X})$ a $G$-invariant klt singularity. Then for any $G$-birational model $f\colon Y\to X$, there will be a $G$-invariant (irreducible) rationally connected subvariety $W\subset f^{-1}(x)$. (In our notation, a point is rationally connected.)
\end{prop}
\begin{proof}In \cite{HX09}, this is shown when $G$ is a Galois group. But the proof does not use any specific property of a Galois group, hence works for any finite group. 
\end{proof}

\begin{lem}\label{gorindex}
 Let $(X,x)$ be a Gorenstein canonical threefold singularity. Assume a finite cyclic group $G$ acts on $(X,x)$.
 Then $\hvol^G(x,X)\leq 27$ with equality if and only if $(X,x)$ is smooth.
\end{lem}

\begin{proof}
Let $\phi:Y\to X$ be a $G$-equivariant maximal crepant model constructed in Lemma \ref{maxcrep}. By Proposition \ref{p-HX}, there exists a $G$-invariant rationally connected closed subvariety $W\subset\phi^{-1}(x)$. Let $\widetilde{W}$ be a $G$-equivariant resolution of $W$. Since $G$ is cyclic and $\widetilde{W}$ is rationally connected, the $G$-action on $\widetilde{W}$ has a fixed point $\tilde{y}$ by holomorphic Lefschetz fixed point theorem. Thus the $G$-action on $W$ has a fixed point $y$ which is the image of $\tilde{y}$. Then Lemma \ref{hvolup} implies that $\hvol^G(x,X)\leq \hvol^G(y,Y)$. Since $(Y,y)$ is a Gorenstein terminal threefold singularity by Lemma \ref{maxcrep}, we know that $(Y,y)$ is a smooth point or an isolated cDV singularity by \cite[Corollary 5.38]{KM98}. Hence $\hvol^G(y,Y)\leq 27$ by Lemma \ref{hypersurf} and we are done.

If the equality holds, then by Corollary \ref{c-strict}, we know that $X=Y$. Then $X$ has to be smooth since otherwise,  $\hvol^G(x,X)\leq 16$. 
\end{proof}




\begin{proof}[Proof of Theorem \ref{t-local}.3]For singularity $x\in X$ and a torsion element of order $r$ in its class group, we can take the corresponding index $1$ cover $y\in Y$. Let $\tilde{y}\in \tilde{Y}$ be the index $1$ cover of $y\in Y$ corresponding to $K_Y$. Then there exists a quasi-\'etale Galois closure $(z\in Z)\to (x\in X)$ of the composite map $(\tilde{y}\in\tilde{Y})\to (x\in X)$ (see e.g. \cite[Theorem 3.7]{GKP16}). In particular, $z\in Z$ is Gorenstein. Denote by $G:=\Aut(Z/X)$ and $G':=\Aut(Z/Y)$. Then $G'$ is a normal subgroup of $G$ with $G/G'\cong\bZ/r$. Pick $g\in G$ whose image in $G/G'$ is a generator. Then Theorem \ref{ghvol} and Lemma \ref{gorindex} imply
\[
\hvol(x,X)=\frac{1}{|G|}\hvol^G(z,Z)\leq \frac{1}{|\langle g\rangle|}\hvol^{\langle g\rangle}(z,Z)\leq \frac{27}{r}.
\]
The proof is finished.
\end{proof}

Next we turn to the proof of Theorem \ref{t-local}.4 which follows from below.


\begin{prop}\label{p-strong}
 Let $(X,x)$ be a Gorenstein canonical threefold singularity whose general hyperplane section is an elliptic singularity, with a nontrivial $G:=\mathbb{Z}/2$-action which only fixes $x\in X$, then $\hvol^G(x,X)\le 18$.
 \end{prop}

\begin{proof}  Donote by $\sigma$ the non-trivial element in $G$. 
Let $\phi_1:Y_1\to X$ be a $G$-equivariant maximal crepant model of $X$ constructed in Lemma \ref{maxcrep}. 
So $Y_1$ is equivariant $\mathbb{Q}$-factorial, i.e., every $G$-invariant divisor is $\mathbb{Q}$-Cartier. By our assumption there is a divisorial part contained in $\phi_1^{-1}(x)$ which we denote by $\Gamma$. Then by running an $-\Gamma$-MMP sequence  over $X$, we can assume $-\Gamma$ is nef over $X$ which implies $\Gamma=\phi_1^{-1}(x)$.

First we make the following reduction.
\begin{lem}We may assume there is a $G(=\mathbb{Z}/2)$-equivariant maximal crepant model $\phi\colon Y\to X$ with an intermediate model $g\colon Y\to Y'$ over $ X$ such that $\rho_G(Y/Y')=1$ and $g(E)$ is a point for $E={\rm Ex}(Y/Y')$. 
\end{lem}
\begin{proof}
 By the proof of Lemma \ref{gorindex} there exists a $G$-invariant closed point $y_1\in Y_1$ over $x$.  If there exists a $G$-invariant $\phi_1$-exceptional component of $\Gamma$ containing $y_1$, we denote this divisor by $E$; otherwise there exist two $\phi_1$-exceptional components of $\Gamma$ containing $y_1$ interchanged by $\sigma$, and we denote the sum of these two divisors by $E$. 
 
 Let us run the $G$-equivariant $(Y_1,\epsilon E)$-MMP over $X$ for $0<\epsilon\ll 1$. If there is a flipping contraction $Y_1\to Y_1'$ that contracts a curve through $y_1$, then $y_1'=($the image of $y_1$ in $Y_1')$ is a $G$-invariant non-$\bQ$-factorial Gorenstein terminal singularity. Since $Y_1'$ is also crepant over $X$, we know that 
 $$\hvol^G(x,X)\leq\hvol^G(y_1', Y_1')\leq 16$$ by Lemma \ref{hvolup} and we are done. Hence we may assume that a flipping contraction $Y_1\to Y_1'$ does not contract a curve through $y_1$. Then the flip $Y_1^+$ is another $G$-equivariant maximal crepant model over $X$ containing a $G$-invariant point $y_1^+$ over $x$. By \cite[1.35]{Kol13} this MMP will terminate as $Y_1\dashrightarrow Y\to Y'$, where $Y_1\dashrightarrow Y$ is the composition of a sequence of flips whose exceptional locus does not contain $y_1$, and $g: Y\to Y'$ contracts the birational transform of $E$, which we also denote by $E$ as abuse of notation. Let $y'$ be the image of $y_1$ in  $Y'$. If $\dim g(E)=1$, then $Y'$ has non-isolated cDV singularities along $g(E)$. Since $y'\in g(E)$ is a $G$-invariant point, Lemma \ref{hypersurf} and \ref{hvolup} implies that $\hvol^G(x,X)\leq \hvol^G(y',Y')\leq 16$ and we are done.

 So the only case left is when $E$ gets contracted by $g$ to a $G$-invariant point $y'\in Y'$.  Since $Y'$ is Gorenstein and crepant over $X$, Lemma \ref{hvolup} implies that $\hvol^G(x,X)\leq \hvol^G(y',Y')$. Therefore, 
 by Lemma \ref{hvolup} it suffices to prove that $\hvol^G(y',Y')\leq 18$.   \end{proof}
 
 \bigskip
 In the above argument, we see that if the indeterminate locus of a flip or its inverse contains a fixed point of $G$, then we know  $\hvol^G(x,X)\le 16$. So we can assume that all $G$-fixed points are contained in the open locus where $Y_1$ and $Y$ are isomorphic. 

Therefore for the choice of $E$, we can make the following assumption which will be needed later.
 \begin{itemize}
 \item[$\clubsuit$]: If there are curves or  divisors in $\Gamma$ which are fixed pointwisely by $G$, we will choose $E$ containing such a curve or a divisor. 
 \end{itemize}
 
 \bigskip
 
 For the rest of the proof, we may assume $(X,x)=(Y',y')$ for simplicity.
 It is clear that $Y$ is a $G$-equivariant maximal crepant model over $X$ as well, so it is Gorenstein terminal and  $G$-$\bQ$-factorial, that is, any $G$-invariant Weil divisor is $\mathbb{Q}$-Cartier. 
Since any  $G$-invariant Weil divisor is $\mathbb{Q}$-Cartier, by local Grothendieck-Lefschetz theorem (see \cite[Section 1]{Rob76}), we know it is indeed Cartier. In particular, the divisor $E$ is Cartier in $Y$, so $E$ is Gorenstein. 
 Since $(-E)$ is $g$-ample, we have $-K_E=-(K_Y+E)|_E=(-E)|_E$ is ample. Hence $E$ is a reduced  Gorenstein del Pezzo surface.
 
 \begin{lem}\label{l-equivirr}Proposition \ref{p-strong} holds if $E$ is irreducible. 
 \end{lem}
 \begin{proof} Since $Y$ has hypersurface singularities, all local embedding dimensions of $E$ are at most $4$. If $E$ is normal, then by similar argument in the proof of Lemma \ref{l-weak} we know that 
 $$\hvol^G(y',Y')\leq\hvol_{Y',y'}(\ord_E)= (-K_E)^2\leq 9.$$ If $E$ is non-normal, then from Reid's classification \cite{Rei94} in the proof of Lemma \ref{l-weak} we know that either $m:=(-K_E)^2\leq 4$ or $E$ is one of the following:
 \begin{itemize}
     \item A linear projection of $F_{m-2;1}$ by identifying a fiber with the negative section;
\item A linear projection of $F_{m-4;2}$ by identifying the negative section to itself via an involution.
 \end{itemize}
Let $\nu:\overline{E}\to E$ be the normalization map. It is clear that $\sigma$ lifts naturally to an involution of the Hirzebruch surface $\overline{E}$ which we denote by $\bar{\sigma}$. We may assume that $\overline{E}$ is not isomorphic to $F_0$ since otherwise $m\leq 4$ and we are done. Hence the negative curve $B$ in $\overline{E}$ is $G$-invariant.

In the first case, $E$ is obtained by gluing a fiber $A$ with the negative curve $B$ from $\overline{E}\cong F_{m-2}$. Since any involution of $\bP^1$ has at least two invariant points, there exists a fixed point $\bar{y}\in B\setminus A$ of $\bar{\sigma}$. If $Y$ is singular at $y:=\nu(\bar{y})$, then we know that $\hvol(y',Y')\leq\hvol(y,Y)\leq 16$ by Lemma \ref{hypersurf} and \ref{hvolup} and we are done. Hence we may assume that $Y$ is smooth at $y$. Denote by $\bar{l}$ the fiber in $\overline{E}$ containing $\bar{y}$ and $l:=\nu_*\bar{l}$. Since $E$ is smooth along $l\setminus\{y\}$ and $E$ is Cartier in $Y$, we know that $Y$ is also smooth along $l\setminus\{y\}$. This implies that $Y$ is smooth along $l$. Besides, we know that $\nu$ is unramified along $\bar{l}$ because $y$ is a normal crossing point of $E$. Thus by similar argument in the proof of Lemma \ref{l-weak}, we have $\cN_{l/Y}^{\vee}\cong\cO_l\oplus\cO_l(2)$ and $\hvol_{Y',y'}(\ord_l)\leq 16$. Since $l$ is $G$-invariant, we have
\[
\hvol^G(y',Y')\leq\hvol_{Y',y'}(\ord_l)\leq 16
\]
and we are done.

In the second case, $E$ is obtained by gluing the negative section $B$ via a non-trivial involution $\tau:B\to B$ from $\overline{E}\cong F_{m-4}$. Let $\bar{y}\in B$ be a fixed point of $\bar{\sigma}$ and $y:=\nu(\bar{y})$. Let $\bar{l}$ be the fiber of $\bar{E}$ containing $\bar{y}$ and $l:=\nu_*\bar{l}$. We may assume that $Y$ is smooth at $y$ (hence smooth along $l$) since otherwise $\hvol^G(y,Y)\leq 16$ and we are done. 
If $\tau(\bar{y})\neq \bar{y}$, then $y$ is a normal crossing point of $E$.  By a similar argument to the previous case, we have $l$ is $G$-invariant, $\cN_{l/Y}^{\vee}\cong\cO_l\oplus\cO_l(2)$ and  $\hvol_{y',Y'}(\ord_{l})\leq 16$ so we are done. If $\tau(\bar{y})=\bar{y}$, then we want to show that $\cN_{l/Y}^{\vee}\cong\cO_l(-1)\oplus\cO_l(3)$.
Let $C$ be the reduced curve with support $\nu(B)$. Then we know that 
\begin{equation}\label{eq-reidlocal}
\cO_E=\ker\left(\nu_*\cO_{\overline{E}}\to(\nu|_{B})_*\cO_B/\cO_C\right)
\end{equation}
by \cite[2.1]{Rei94}.
Let us pick a system of parameters $(u_1,u_2)$ of the local
ring $\cO_{\overline{E},\bar{y}}$ such that the local
equation $B$ is $(u_1=0)$ and the involution $\tau:B\to B$
satisfies $\tau^{*}(u_2)=-u_2$. After taking completions,
we get
\[
 \widehat{\cO_{\overline{E},\bar{y}}}=\bC[[u_1,u_2]],\quad
 \widehat{\cO_{B,\bar{y}}}=\bC[[u_1,u_2]]/(u_1),\quad
 \widehat{\cO_{C,y}}=\bC[[u_1,u_2^2]]/(u_1). 
\]
Hence $\widehat{\cO_{E,y}}=\bC[[u_1, u_1 u_2,u_2^2]]$
by \eqref{eq-reidlocal}. Denote by $(x_1,x_2,x_3):=(u_2^2, u_1, u_1 u_2)$,
then 
\[
\widehat{\cO_{E,y}}\cong\bC[[x_1,x_2,x_3]]/(x_1 x_2^2-x_3^2). 
\]
This implies that $E$ has a pinch point at $y$. From these local
equations, we know that the scheme-theoretic fiber of $\nu$ over $y$ is isomorphic to $\Spec~\bC[x]/(x^2)$. Hence in an open neighborhood $U$ of $\bar{y}$, the sheaf $\Omega_{\overline{E}/E}|_U\cong\bC_{\bar{y}}$ is a skyscraper sheaf supported at $\bar{y}$ of
length $1$. We have the following exact sequence
\[
 \nu^*\Omega_{E}\otimes\cO_{\bar{l}}\to\Omega_{\overline{E}}\otimes
 \cO_{\bar{l}}\to\Omega_{\overline{E}/E}\otimes\cO_{\bar{l}}\to 0.
\]
Denote by $\cF:=\mathrm{image}(\Omega_{E}\otimes\cO_{l}\to\nu_*\Omega_{\overline{E}}\otimes
 \cO_{l})$, then we have the following commutative diagram
  \[
 \begin{tikzcd}
  0\arrow{r}&\cF\arrow{r}\arrow[two heads]{d} & 
  \nu_*\Omega_{\overline{E}}\otimes\cO_l\arrow{r}\arrow[two heads]{d}
  {}& \bC_{y}\arrow{r} & 0\\
 0\arrow{r}&\Omega_l \arrow{r}{\cong}& \Omega_l\arrow{r} & 0 &
 \end{tikzcd}
 \]
 where the horizontal sequences are exact.
Thus we have a short exact sequence by taking kernels of vertical maps:
\[
 0\to\ker(\cF\to\Omega_l)\to\nu_*\cN_{\bar{l}/\overline{E}}^{\vee}\to
 \bC_y\to 0.
\]
Since $\cN_{\bar{l}/\overline{E}}^{\vee}\cong\cO_{\bar{l}}$, we know that 
$\ker(\cF\to\Omega_l)\cong\cO_l(-1)$. From the following surjective sequence
\[
\Omega_{Y}\otimes\cO_l\twoheadrightarrow\Omega_E\otimes\cO_l\twoheadrightarrow\cF\twoheadrightarrow\Omega_l
\]
we know that there is a surjection
from $\cN_{l/Y}^{\vee}$ to $\cO_{l}(-1)$. Since $\deg(\cN_{l/Y}^{\vee})=2$, this
surjection splits since $\Ext^1(\cO_l(-1),\cO_l(3))=0$. Hence $\cN_{l/Y}^{\vee}\cong
\cO_l(-1)\oplus\cO_l(3)$. 
 Since 
\[
\cI_l^k/\cI_l^{k+1}\cong\Sym^k\cN_{l/Y}^{\vee}\cong\cO_l(-k)\oplus\cO_l(-k+4)\oplus\cdots\oplus\cO_l(3k),
\]
we know that $h^0(l,\cI_l^k/\cI_l^{k+1})=\frac{9}{8}k^2+O(k)$. Since there is an injection 
$$(\fa_k/\fa_{k+1})(\ord_{l})\hookrightarrow H^0(l,\cI_l^k/\cI_l^{k+1}),$$ we have $\ell(\cO_{Y',y'}/\fa_k(\ord_l))\leq \frac{3}{8}k^3+O(k^2)$ which implies  $\vol_{Y',y'}(\ord_l)\leq \frac{9}{4}$. Since $A_{Y'}(\ord_l)=2$, we have $$\hvol^G(y',Y')\leq\hvol_{Y',y'}(\ord_l)=8\cdot \vol_{Y',y'}(\ord_l)\leq 18.$$
\end{proof}

  \begin{lem}Proposition \ref{p-strong} holds when $E=E_1+E_2$  is reducible.
\end{lem}
\begin{proof} 
Denote by $\fb_k:=g_*\cO_Y(-kE)$
as an ideal of $\cO_{Y',y'}$. Since $(-E)$ is $g$-ample, for $k$ sufficiently divisible we have
\[
 \lct(\fb_k)\leq \frac{A_{Y'}(E_1)}{\ord_{E_1}(\fb_k)}=\frac{1}{k},\qquad\mult(\fb_k)=(E^3)\cdot k^3.
\]
In particular, we have 
\[
 \hvol^G(y',Y')\leq \lct(\fb_k)^3\cdot \mult(\fb_k)\leq
(E^3)=(-K_E)^2.
\]
From Reid's classification \cite{Rei94} and the fact that $\sigma$ induces an isomorphism between $(E_1,\omega_E|_{E_1})$ and $(E_2,\omega_E|_{E_2})$, either $(-K_E)^2\leq 8$ or $E$ is one of the following :
\begin{itemize}
 \item $E$ is obtained by gluing two copies of $F_{a;0}$ along their line pairs/double lines;
 \item $E$ is obtained by gluing two copies of $F_{a;1}$ 
 along their line pairs consisting of one fiber and negative section; 
 \item $E$ is obtained by gluing two copies of $F_{a;2}$ along their negative sections.
\end{itemize}

We recall the definition of $F_{a;0}$ in \cite{Rei94}. For
any integer $a>0$, the surface $F_{a;0}\subset\bP^{a+1}$ is the cone
over a rational normal curve of degree $a$. A line pair
on $F_{a;0}$ means the union of two different generators
of the cone. A double line on $F_{a;0}$ is the Weil divisor
of twice a generator. 

\medskip

In the first case, since the local embedding dimension of $E_i\cong F_{a;0}$ at the cone point is at most $4$, we have $a\leq 3$. Hence $(-K_E)^2=2a\leq 6$ and we are done.

\begin{rem} We want to remark, up to this point, we do not use the assumption that $x$ is the only fixed point of $G$ on $X$. 
\end{rem}

In the second case, denote the line pairs on $E_i$ by 
$A_i\cup B_i$ where $A_i$ and $B_i$ are fiber and negative
section, respectively. Then $E$ is obtained by gluing
these two line pairs via a gluing isomorphism $\tau:A_1\cup
B_1\to A_2\cup B_2$. Since $(-K_E)^2=2a+4$, we may assume that
$a>0$. Denote by $\nu:\overline{E}\to E$ the normalization of $E$.
Then we know that $\bar{\sigma}(B_1)=B_2$
and $\bar{\sigma}(A_1)=A_2$ with $\bar{\sigma}$ the lifting
of $\sigma$ to $\overline{E}$. 
We separate to two subcases below.

\textbf{Subcase 1}: $\tau(A_1)=B_2$ and $\tau(A_2)=B_1$.
Thus $\sigma$ interchanges $\nu(A_1)$ and $\nu(B_1)$ 
which are the irreducible components of the non-normal locus of $E$.
Hence in $E$ there is only one $\sigma$-fixed point $y$ which is the image of the intersection point of the pair of lines on $E_i$.  We may assume $y$ is smooth, since otherwise 
$\hvol^G(x,X)\le \hvol^G(y,Y)\le 16.$ Then the (common) tangent plane $\Theta$ of $E_i$ is fixed by $\sigma$, and on it $ \sigma$ interchanges  two smooth lines which are the tangent lines of the pair of lines. Thus the action is either of type $\frac{1}{2}(0,1,1)$ or $\frac{1}{2}(0,0,1)$. Therefore, there is a divisor or a curve on $Y$ that is fixed pointwisely by $G$.
By our assumption of Proposition \ref{p-strong}, it is contained in $\Gamma$ but $E$ does not contain any such curve. Thus it is contradictory to our Assumption $\clubsuit$.  

\textbf{Subcase 2}: $\tau(A_1)=A_2$ and $\tau(B_1)=B_2$.
Denote by $A:=\nu(A_1)=\nu(A_2)$ and $B:=\nu(B_1)=\nu(B_2)$.
Thus $\sigma$ preserves $A$ and $B$. If neither
$\sigma|_{A}$ nor $\sigma|_{B}$ is an
identity, then there is a $\sigma$-fixed point $z\in E$
such that $E$ is normal crossing at $z$. Since $\sigma$ 
interchanges two irreducible components of $E$, locally at $z$ the
$G$-action is of type $\frac{1}{2}(0,1,1)$ or $\frac{1}{2}(0,0,1)$.
Therefore, there is a divisor or a curve on $Y$ that is fixed
pointwisely by $G$ and contains $z$. But any such curve through $z$
is not contained in $E$ since $\sigma|_E$ only has isolated fixed point. Thus it contradicts
to our Assumption $\clubsuit$. So we may assume that 
$\sigma$ fixes $A$ or $B$. 
Assume first $\sigma|_B$ is an identity. 
By the same reason in subcase 1, we may assume that $Y$ is 
smooth along $B$.
Consider the following map:
\[
 \Phi:(\cT_{E_1}\oplus\cT_{E_2})\otimes
 \cO_{B}\to\cT_Y\otimes\cO_{B}.
\]
By the above discussion, we can pick an analytic local coordinate $(x_1,x_2,x_3)$ on $Y$ with
origin at $y=A\cap B$ such that $E_1=(x_3=0)$, $E_2=(x_3=x_1x_2)$,
$A=(x_1=x_3=0)$ and $B=(x_2=x_3=0)$. Then we have
$\cT_{E_1}\otimes\cO_B=\langle \partial_{x_1},\partial_{x_2}\rangle$
and $\cT_{E_2}\otimes\cO_B=\langle\partial_{x_1},\partial_{x_2}+x_1 \partial_{x_3}\rangle$.
Thus $\mathrm{Im}(\Phi)=\langle\partial_{x_1}, \partial_{x_2},
x_1\partial_{x_3}\rangle$. This implies that the cokernel
of $\Phi$ near $y$ is a skyscraper sheaf $\bC_{y}$. Therefore,
we have an exact sequence
\[
 0\to \cT_B\to(\cT_{E_1}\oplus\cT_{E_2})\otimes
 \cO_{B}\to\cT_Y\otimes\cO_{B}\to \bC_y\to 0.
\]
Taking degrees of the above exact sequence, we get
\[
 \deg\cT_B+\deg\cT_Y\otimes\cO_B=\deg\cT_{E_1}\otimes\cO_B+
 \deg\cT_{E_2}\otimes\cO_B+1.
\]
which implies $\deg\cN_{B/Y}=\deg\cN_{B/E_1}+\deg\cN_{B/E_2}+1$.
Since $Y$ is crepant over $X$, we know $(K_Y\cdot B)=0$ which implies $\deg\cN_{B/Y}=-2$. We also know that 
$\deg\cN_{B/E_1}=\deg\cN_{B/E_2}=-a$. Hence $-2=-2a+1$, but 
this is absurd since $a$ is an integer. We get a contradiction. Now we assume $\sigma$ fixes $A$ pointwisely. By the exactly the same calculation as above, but replacing $B$ by $A$, we know that 
 $\deg\cN_{A/Y}=\deg\cN_{A/E_1}+\deg\cN_{A/E_2}+1$, where the left hand side is  $-2$, but the right hand side is 1. This is again absurd. 
\medskip

For the third case, denote by $B:=E_1\cap E_2$ and it is invariant under $\sigma$. If $\sigma$ fixes every point on $B$, then we can assume $Y$ is smooth along $B$ as otherwise we will have $\hvol(y,Y)\le 16$ for a singular fixed point $y$ on $B$. This implies that $Y$ is smooth along $E$, as $E_i$ are smooth Cartier divisors along $E_i\setminus B$.  We have the following short exact sequence:
\[
 0\to\cT_{B}\to(\cT_{E_1}\oplus\cT_{E_2})\otimes\cO_B\to\cT_{Y}
 \otimes\cO_B\to 0.
\]
Hence we know that $\bigwedge^2\cN_{B/Y}\cong\cN_{B/E_1}\otimes\cN_{B/E_2}$.
Since $\deg\cN_{B/Y}=-2$ and $\deg\cN_{B/E_i}=-a$, we have $a=1$. Hence $(-K_E)^2=2(a+2k)=10$ and we are done. The remaining case is that $\sigma$ induces a nontrivial order 2 automorphism on $B$. Denote by $y\in B$ be a fixed point, then we can again assume it is a smooth point. Since $\sigma$ interchanges $E_1$ and $E_2$, similar as before, we know either $y\in Y$ is not smooth or $y\in Y$ is smooth but $\Gamma$ contains a curve which is pointwisely fixed by $\sigma$ but not $E$. Thus the latter case is again contradictory to our Assumption $\clubsuit$. 
\end{proof}
Thus we finish the proof of Proposition \ref{p-strong}. 
\end{proof}

\begin{proof}[Proof of Theorem \ref{t-local}.4]
Let $L$ be a divisor whose class in $\mathrm{Pic}(x\in X)$ is nontrivial torsion. Let $(\tilde{x}\in \tilde{X})$ be  the index 1 cover with respect to $L$. By Theorem \ref{t-local}.3, we can assume the index of $L$ is 2. 

If $\tilde{x}\in \tilde{X}$ is not Gorenstein, then $x\in X$ is not Gorenstein either. The index of $K_X$ is $2$ by Theorem \ref{t-local}.3. Let $\pi:(y\in Y)\to(x\in X)$ be the index $1$ cover with respect to $K_X$. If $\pi^*L$ is Cartier at $y$, then $L\sim K_X$ and hence $(\tilde{x}\in\tilde{X})\cong(y\in Y)$ is Gorenstein which is a contradiciton. So $\pi^*L$ is not Cartier at $y$, and we may replace $(x\in X,L)$ by $(y\in Y,\pi^*L)$ since $\hvol(x,X)<\hvol(y,Y)$ by Lemma \ref{ghvol}. Therefore, we can always assume that $(\tilde{x}\in\tilde{X})$ is Gorenstein.

By Lemma \ref{ghvol} it suffices to show $\hvol^G(\tilde{x},\tilde{X})\leq 18$. The covering $(\tilde{x}\in\tilde{X})$ is not smooth as we assume $(x\in X)$ is not a quotient singularity. If $(\tilde{x}\in\tilde{X})$ is cDV, then $\hvol^G(\tilde{x},\tilde{X})\le 16$ by Lemma \ref{hypersurf}.
Thus we can assume a general section of $\tilde{x}\in \tilde{X}$ is of elliptic type, then $\hvol^G(\tilde{x},\tilde{X})\le 18$ by Proposition \ref{p-strong}. Hence the proof is finished.
\end{proof}

\subsection{Effective bounds on local fundamental groups}

\begin{lem}\label{klein}
 Let $G$ be a finite group acting faithfully on $\bP^1$, then the smallest $G$-orbit has at most $12$ points, which only happens when $G=A_5$.
\end{lem}

\begin{proof}
 This follows from the classification of finite subgroups of ${\rm SL}(2,\bC)$ (see
 \cite{Kle93}).
\end{proof}

\begin{prop}\label{l-fundgp}
Let $(X,x)$ be a Gorenstein canonical threefold singularity  with a finite group $G$-action. Then $\hvol^G(x,X)< 324$.
\end{prop}

\begin{proof}
 We may assume that $(X,x)$ is not a hypersurface singularity since otherwise we are done by Lemma \ref{hypersurf}.
 Let $\phi_1: Y_1\to X$ be a $G$-equivariant maximal crepant model of $X$ constructed in Lemma \ref{maxcrep}. By Proposition \ref{p-HX}, there exists a $G$-invariant rationally connected closed subvariety $W\subset\phi_1^{-1}(x)$. If $W=\{y_1\}$ is a closed point, then Lemma \ref{hypersurf} and \ref{hvolup} imply $\hvol^G(x,X)\leq\hvol^G(y_1,Y_1)\leq 27$ and we are done. If $W$ is a rational curve, then by Lemma \ref{klein} there exists a closed point $y_1\in W$ whose $G$-orbit has at most $12$ points. Let $H\subset G$ be the group of stabilizer of $y_1$, then Lemma \ref{hypersurf} and \ref{hvolup} implies that $\hvol^H(y_1,Y_1)\leq 27$. Since $[G:H]\leq 12$, by Theorem \ref{ghvol} we know that $$\hvol^G(x,X)< 12\hvol^H(x,X)\leq 12\hvol^H(y_1,Y_1)\leq 324.$$
 
 Hence we may assume that $W=E$ is a rational surface. Let us run the $G$-equivariant $(Y_1,\epsilon E)$-MMP over $X$ for $0\leq \epsilon\ll 1$. By \cite[1.35]{Kol13} this MMP will terminate as $Y_1\dashrightarrow Y\to Y'$, where $Y_1\dashrightarrow Y$ is the composition of a sequence of flips, and $g:Y\to Y'$ contracts the birational transform of $E$ which we also denote by $E$ as abuse of notation. If $\dim g(E)=1$, then $Y'$ has non-isolated cDV singularities along $g(E)$. Since $g(E)$ is a $G$-invariant rational curve, there exists $y'\in Y'$ whose $G$-orbit has at most $12$ points by Lemma \ref{klein}. Let $H\subset G$ is the group of stabilizer of $y'$, if $H\neq G$, then we have $$\hvol^G(x,X)< 12\hvol^H(x,X)\leq 12\hvol^H(y',Y')\leq 192.$$ Otherwise, $H=G$ and we have $\hvol^G(x,X)=\hvol^H(x,X)\le \hvol^H(y',Y')\leq 27$.
 
 Now the only case left is when $g(E)=y'$ is a $G$-invariant closed point. If $E$ is normal, then by similar argument in the proof of Lemma \ref{l-weak} we know that $$\hvol^G(x,X)\leq\hvol^G(y',Y')\leq \hvol_{y',Y'}(\ord_E)=(-K_E)^2\leq 9$$ and we are done. If $E$ is non-normal, then from Reid's classification \cite{Rei94} we know that either $\hvol^G(x,X)\leq m:=(-K_E)^2\leq 4$ or $E$ is one of the following:
 \begin{itemize}
     \item A linear projection of $F_{m-2;1}$ by identifying a fiber with the negative line;
\item A linear projection of $F_{m-4;2}$ by identifying the negative conic to itself via an involution.
 \end{itemize}
 In both cases above the non-normal locus $C$ of $E$ is a rational curve. Since $C$ is $G$-invariant, Lemma \ref{klein} implies that there exists a closed point $y\in C$ whose $G$-orbit has at most $12$ points. Then we have $$\hvol^G(x,X)< 12\hvol^H(x,X)\leq 12\hvol^H(y,Y)\leq 324,$$ where $H\subset G$ is the group of stabilizer of $y$. So we finish the proof.
\end{proof}

\begin{proof}[Proof of Theorem \ref{fundgp}]
 We will show that $|\locpione(X,x)|\cdot\hvol(x,X)< 324$ (note that this gives a new proof that the algebraic fundamental group of a three dimensional klt singularity is finite (see  \cite{SW94, Xu14}). It suffices to show the inequality $$|\Aut(\tilde{X}/X)|\cdot\hvol(x,X)< 324$$
 for any finite  quasi-\'etale Galois morphism $\pi:(\widetilde{X},\tilde{x})\to(X,x)$. By taking  Galois closure of index one covering of $(\widetilde{X},\tilde{x})$, we may assume that $(\widetilde{X},\tilde{x})$ is Gorenstein. Denote by $G:=\Aut(\widetilde{X}/X)$, then Theorem \ref{ghvol} and Proposition \ref{l-fundgp} imply that $$|G|\cdot\hvol(x,X)=\hvol^G(\tilde{x},\widetilde{X})< 324.$$ Hence we have shown $|\locpione(X,x)|\cdot\hvol(x,X)< 324$. Then the proof follows from \cite[Corollary 1.4]{TX17} which asserts that $\pi_1(\Link(x\in X))$ is finite (hence isomorphic to its pro-finite completion $\locpione(X,x)$).
\end{proof}

\subsection{K-moduli of cubic threefolds as GIT}
In this section, we give a brief account on how Theorem \ref{t-local} implies Theorem \ref{t-global}. Such an argument for surface appeared in \cite{OSS16}, and  was also sketched in \cite{SS17} for cubic hypersurfaces. 

A straightforward consequence from \cite{Liu16} is the following result.
\begin{lem}\label{l-cubic}Let $X$ be a $\mathbb{Q}$-Gorenstein smoothable K-semistable Fano varieties, such that its smoothing is a cubic threefold. Then
the local volume of any point on $X$ is at least $81/8$.
\end{lem}
\begin{proof} The volume of $(X,-K_X)$ is $(-K_X)^3=24$. Then by \cite[Theorem 1]{Liu16}, we know the local volume of a point $x\in X$ satisfies that
$$\hvol(x,X)\ge 24\times (3/4)^3=81/8.$$
\end{proof}
\begin{lem}\label{l-cubicdp}Let $X$ be a $\mathbb{Q}$-Gorenstein smoothable K-semistable Fano variety, such that its smoothing is a cubic threefold. Then $X$ is Gorenstein, furthermore, $-K_X=2L$ for some Cartier divisor $L$.
\end{lem}
\begin{proof}
We first show that $X$ is Gorenstein.
Assume to the contrary that $X$ is not Gorenstein at a point
$x\in X$. By Lemma \ref{l-cubic} and Theorem \ref{t-local}.3, we know its Cartier index is equal to 2.
If $x\in X$ is an isolated singularity, we choose a neighborhood
$U$ of $x$ such that $U\setminus\{x\}$ is smooth.
If $x\in X$ is not an isolated singularity, then
since a klt variety has only quotient singularities
in codimension $2$, we can choose a neighborhood 
$U\subset X$ of $x$ such that $U\setminus \{x\}$ has 
only non-isolated quotient singularities. Then by Lemma \ref{l-cubic}
and Theorem \ref{t-local}.2, $U\subset X$ has only quotient
singularities of type $\frac{1}{2}(1,1,0)$. Thus
in any case we have an open neighborhood $U$ of $x$ such that
$U\setminus\{x\}$ is Gorenstein. In particular, $K_X$ is 
a nontrivial torsion element in $\mathrm{Pic}(x\in X)$ since it is Cartier
in a punctured neighborhood of $x$. 
If $x\in X$ is a quotient singularity, then by Lemma \ref{l-cubic}
and Theorem \ref{t-local}.2 it can only 
be of order $2$ hence of type $\frac{1}{2}(1,1,1)$.
But this is a contradiction since $x\in X$ is not smoothable by \cite{Sch71}.
Hence $x\in X$ is not a quotient singularity, which implies
$\hvol(x,X)\leq 9$ by Theorem \ref{t-local}.4. But this contradicts
Lemma \ref{l-cubic}. As a result, $X$ has only Gorenstein 
canonical singularities. 

Let $\mathcal{X}\to C$ be a family which gives a $\mathbb{Q}$-Gorenstein deformation of $X$ to some smooth cubic threefold such that over $0\in C$, $\mathcal{X}_0\cong X$. By shrinking $C$, we can assume $\mathcal{X}^0=\mathcal{X}\times_C (C\setminus \{0\})$ is a family of smooth cubic threefolds. Then $-K_{\mathcal{X}^0}\sim_{C^0}\mathcal{O}(2)$. By taking the closure, we know $-K_{\mathcal{X}}\sim_C 2\mathcal{L}$ for some $\bQ$-Cartier integral Weil divisor $\mathcal{L}$. By inversion of adjunction (see \cite[Corollary 1.4.5]{BCHM10}), we know that $\cX$ has Gorenstein canonical singularities. We want to show that $\cL$ is in fact Cartier.

Assume to the contrary that $\cL$ is not Cartier at $x\in X$. Since $\cO_{\cX}(\cL)$ is Cohen-Macaulay by \cite[5.25]{KM98}, we know that $\cO_{\cX}(\cL)\otimes\cO_X$ is $S_2$. Hence $L:=\mathcal{L}|_X$ is a $\bQ$-Cartier integral Weil divisor on $X$ satisfying $\cO_X(L)\cong \cO_{\cX}(\cL)\otimes\cO_X$. Thus $L$ can not be Cartier at $x$ since otherwise $\cO_{\cX}(\cL)$ would be locally free at $x$.  If $x\in X$ is a non-smooth quotient singularity satisfying $\hvol(x,X)\geq 81/8$ then the index has to be 2 by Theorem \ref{t-local}.2. If moreover it is not of type $\frac{1}{2}(1,1,1)$, then it is of type $\frac{1}{2}(1,1,0)$, i.e. locally analytically defined by $(x_1^2+x_2^2+x_3^2=0)$ in $(0\in\bA^4)$.
Similarly, if $x\in X$ has a neighborhood $U$ such that  $x$ is the only non-Cartier point on $L|_U$, then Theorem \ref{t-local}.4 implies that $\hvol(x,X)\leq 9$ which contradicts Lemma \ref{l-cubic}. Therefore, $L$ is not Cartier along a curve $C\subset X$. In particular, $C$ is contained in the singular locus of $X$ and we can replace $x$ by a general point $(x'\in C\subset X)$ which is of quotient type as any general singularity along a curve on a klt threefold.
This again implies it is of type $\frac{1}{2}(1,1,0)$ by the previous argument. 

Since $\edim(X,x)=4$, we know that $(x\in \cX)$ is a hypersurface singularity as well. Let $H$ be a general hyperplane section of $\cX$ through $x$. Then $(x\in H)$ is a normal isolated hypersurface singularity. By similar arguments, there is a well-defined $\bQ$-Cartier integral Weil divisor $\cL|_{H}$ on $H$ satisfying $\cO_H(\cL|_{H})\cong\cO_{\cX}(\cL)\otimes\cO_H$. By local Grothendieck-Lefschetz theorem (see \cite{Rob76}), the local class group of $(x\in H)$ is torsion free which implies that $\cL|_{H}$ is Cartier at $x$. Hence $\cL$ is Cartier at $x$ and we get a contradiction. As a result, the Weil divisor $L=\cL|_{X}$ is  Cartier and $-K_X=2L$ by adjunction.
\end{proof}

\begin{lem}\label{l-embed}
Let $X$ be a $\mathbb{Q}$-Gorenstein smoothable K-semistable Fano varieties, such that its smoothing is a cubic threefold. Then $X$ is a cubic threefold in $\mathbb{P}^4$.
\end{lem}
\begin{proof} By Lemma \ref{l-cubicdp}, we have that $X$
is a del Pezzo variety (see \cite[P 117]{Fuj90} for the definition) 
of degree $3$ and thus $L$ is very ample by \cite[Section 2]{Fuj90}. 
\end{proof}

\bigskip

\begin{proof}[Proof of Theorem \ref{t-global}]We know at least one smooth cubic threefold, namely the Fermat cubic threefold, admits a KE metric (see \cite{Tia87} or \cite[p 85-87]{Tia00}). Let $\mathbb{P}^{34}$ be the space parametrizing all cubic threefolds. By \cite{LWX14}, there is an artin stack $ \mathcal{M}$ containing an Zariski open set of $[\mathbb{P}^{34}/{\rm PGL}(5)]$ such that the $\mathbb{C}$-points of $\mathcal{M}$ parametrize the isomorphic classes of K-semistable Fano threefolds which can smoothed to a smooth cubic threefold. Moreover, there is a morphism $\mu\colon \mathcal{M}\to M$ which yields a good quotient morphism such that $M$ is a proper algebraic scheme whose close points parametrizes isomorphic classes of K-polystable ones.  

Lemma \ref{l-embed} then shows that all points in $\mathcal{M}$ indeed parametrize cubic threefolds. Then by a result of Paul-Tian (see \cite{Tia94} or \cite[Corollary 3.5]{OSS16}), we know they are all contained in the locus $U^{\rm ss}$ of GIT semistable cubic threefolds. Denote the GIT quotient  
$(U^{ss}\subset \mathbb{P}^{34})\to M^{\rm GIT}$.
To summarize, we obtain an morphism $g\colon \mathcal{M}\to [U^{\rm ss}/{\rm PGL}(5)]$, whose good quotient yields a morphism $h\colon M\to M^{\rm GIT}$.

We then proceed to show that the morphism $g$ is an isomorphism for which we only need to verify it is bijective on $\mathbb{C}$-points.
The injectivity follows from the modular interpretation. The surjectivity on the polystable points follow from the fact that $M$ is proper.   And this indeed implies  the surjectivity  on the semistable points since $\mathcal{M}$ consists of all cubic threefolds whose orbit closures contain K-polystable points, which are then precisely the GIT semistable points. 
  \end{proof}

\begin{proof}[Proof of Corollary \ref{c-list}] 
The list of GIT-(polystable)stable three dimensional cubics is provided by the main results in \cite[Section 1]{All03}. Since they are polystable, then the existence of KE metric follows from \cite{CDS15,Tia15}.
\end{proof}
\begin{rem}We want to remark that our proof of Theorem \ref{t-global} only uses two pieces of simple information of cubic threefolds: its volume and Picard group. So we expect such a strategy with Theorem \ref{t-local} can be used to construct many other compact K-moduli of smoothable threefolds. 
\end{rem}

\section{Discussions}

There are  lots of interesting questions on the volume of a singularity. Here we mention some of them which are related to our work. 

As we mentioned before, we expect the following to be true, whose proof would simplify and strength our result. 
\begin{conj}\label{c-finite}
Let $f\colon (x\in X)\to (y\in Y)$ be a quotient map of klt singularities by the group $G$, which is \'etale in codimension 1, then
$$\hvol(y, Y)\cdot |G|=\hvol(x, X).$$
\end{conj}

As mentioned in Theorem \ref{ghvol}.3  this is known in the 
quasi-regular case, i.e.,  if $\hvol(x, X)$ is computed by a divisorial valuation. 
In particular, Conjecture \ref{c-finite} holds in dimension $2$
by \cite[Proposition 4.10]{LL19}.

\bigskip

Next we turn to the set (contained in $(0,27]$ as we show) of volumes of three dimensional klt singularities.

\begin{expl}
Let $(X_{p,q},x)$ be the singularity in $(\bA^4, o)$ defined by $x^2+y^2+z^p+w^q=0$. If $p,q\geq 2$, $2p>q$ and $2q>p$, then $(X_{p,q},x)$ is a quasi-regular canonical singularity with minimizing valuation $v_*$ of weights $(pq,pq,2q,2p)$ (see \cite{CZ15}). Thus computation shows 
\[
\hvol(x, X_{p,q})=\hvol(v_*)=\frac{4(p+q)^3}{p^2 q^2}.
\]
These volumes $\hvol(x, X_{p,q})$ are definitely discrete away from $0$.
\end{expl}

We may ask the following question:

\begin{que}
Is the set of volumes of $3$-fold klt singularities discrete away from $0$?
\end{que}

Actually the same question can be asked for any dimension, though we do not have a lot of evidence. Even for quasi-regular cases it seems to be a hard question. 

\begin{expl}
Let $V$ be a K-semistable klt log del Pezzo surface. Let $q$ be the largest integer such that there exists a Weil divisor $L$ satisfying $-K_V\sim_{\bQ} qL$. Then $(X,o):=(C(V,L),o)$ is a threefold klt singularity. By \cite{LX16} we know that 
\[
\hvol(o, X) = \hvol_{o, X}(\ord_{V}) = q(-K_V)^2.
\]
From discussions above we know that $q(-K_V)^2\leq 27$. Discreteness of $\hvol$ would imply that $q = o(\frac{1}{(-K_V)^2})$.
\end{expl}

\bigskip


We also have the following conjecture on the singularities with
large volumes in general dimension. 
\begin{conj}\label{c-largev}

The second largest volume of $n$-dimensional klt singularity is $2(n-1)^n$, and it reaches this volume if and only if it is an ordinary double point. 
\end{conj}
 This conjecture is asked in \cite{SS17}. We confirmed Conjecture \ref{c-largev} when dimension is at most 3 in Theorem \ref{t-local}. As we mentioned (see Remark \ref{rem-ss}), \cite{SS17} shows that this conjecture in dimension $n$ together with the finite degree formula (see Conjecture \ref{c-finite}) implies in dimension $n$, the K-stable moduli space coincides with the GIT moduli space for cubic hypersurfaces.  We also note that we indeed only need the finite degree formula holds for a singularity $(y\in Y)$ appearing on the Gromov-Hausdorff limit of K\"ahler-Einstein Fano varieties \footnote{This is recently confirmed in \cite[Theorem 1.7]{LX17}}.

\appendix

\section{Optimal bounds of volumes of singularities}
\label{hvolbound}

In this appendix we will prove Theorem \ref{t-hvolbound}
and its logarithmic version Theorem \ref{t-loghvolbound}. Let us begin with proving the inequality part.

For a morphism $\pi:\cX\to C$ from a variety $\cX$ to a smooth curve $C$ (over
$\bC$), we say an ideal sheaf $\fa$ on $\cX$ is a \emph{flat family of ideals over} $C$ if the 
quotient sheaf $\cO_{\cX}/\fa$ is flat over $C$. For an $n$-dimensional
klt pair $(X,\Delta)$ and a closed point $x\in X$, we define the \emph{volume
of the singularity} $x\in (X,\Delta)$ to be 
$$\hvol(x,X,\Delta):=\min_{v\in\Val_{X,x}}
\hvol_{(X,\Delta),x}(v)$$
where $\hvol_{(X,\Delta),x}(v):=A_{(X,\Delta)}(v)^n\vol(v)$ (notice
that such minimum exists by \cite{Blu16}).

\begin{lem}\label{hvolbound-ineq}
 Let $x\in (X,\Delta)$ be an $n$-dimensional klt singularity.
 Then $\hvol(x,X,\Delta)\leq n^n$.
\end{lem}

\begin{proof} 
 Let $C_1\subset X$ be a curve through $x$
 that intersects $X_{\mathrm{reg}}\setminus\Supp(\Delta)$.
 Denote by $\tau:\bar{C}_1\to C_1$ the normalization of $C_1$.
 Pick a point $0\in \tau^{-1}(x)$, then there exists a Zariski open neighborhood
 $C$ of $0$ in $\bar{C}_1$ such that $\tau(C\setminus\{0\})\subset X_{\mathrm{reg}}\setminus\Supp(\Delta)$.
 Then $\pr_2:\cX:=X\times C\to C$ has a section $\sigma=(\tau,\id):C\to\cX$.
 Denote by $C^\circ:=C\setminus\{0\}$. Let $\fa$ be the ideal sheaf on $\cX$
 defining the scheme theoretic image of $\sigma$. Since $\tau(c)$ is a smooth
 point of $\cX$ for any $c\in C^{\circ}$, we know that $\cO_{\cX_c}/(\fa^m)_c
 \cong \cO_{X,\tau(c)}/\fm_{\tau(c)}^m$ has constant length $\binom{n+m-1}{n}$.
 Hence $(\fa^m)^{\circ}$ is a flat family of ideals over $C^{\circ}$.
 Thus there exists a unique ideal sheaf
 $\fb_m$ on $\cX$ which extends $(\fa^m)^{\circ}$ to a flat family of 
 ideals over $C$ by \cite[Proposition III.9.8]{Har77}. For $i+j=m$, applying
 \cite[Proposition III.9.8]{Har77} to $V(\fb_i\fb_j)\to C$ implies 
 that $\fb_m\supset\fb_i\fb_j$, hence $(\fb_\bullet)$ is a graded sequence
 of flat families of ideals over $C$. Denote by $\Delta_c$ the pushforward
 of $\Delta$ under the isomorphism $X\to\cX_c$, then $\sigma(c)\not\in\Supp(\Delta_c)$
 for all $c\in C^{\circ}$ by our assumption.
 For a general $c\in C^{\circ}$, flatness of $\fb_\bullet$ implies
 \begin{align*}
  \ell(\cO_{\cX_{0}}/(\fb_{m})_{0})&=\ell(\cO_{\cX_{c}}/(\fb_{m})_{c})
  =\ell(\cO_{X,\tau(c)}/\fm_{\tau(c)}^m)=\binom{m+n-1}{n};\\
  \lct(\cX_{0},\Delta_{0};(\fb_{m})_{0})&\leq \lct(\cX_{c},\Delta_c;(\fb_{m})_{c})
  =\lct(X;\fm_{\tau(c)}^m)=\frac{n}{m}.
 \end{align*}
 Here the inequality on $\lct$'s follows from the lower semi-continuity
 of log canonical thresholds (see e.g. \cite[Corollary 9.5.39]{Laz04b} and
 \cite[Proposition A.3]{Blu16}).
 Since $(\cX_{0},\sigma(0))\cong(X,x)$, we have 
 \begin{align*}
  \hvol(x,X,\Delta)&=\hvol(\sigma(0),\cX_{0})
  \leq \lct(\cX_{0},\Delta_{0};(\fb_{\bullet})_{0})^n\mult((\fb_{\bullet})_{0})\\
  &=n! \lim_{m\to\infty}\lct(\cX_{0},\Delta_{0};(\fb_{m})_{0})^n\cdot
  \ell(\cO_{\cX_{0}}/(\fb_{m})_{0})\\
  &\leq n!\lim_{m\to\infty} \left(\frac{n}{m}\right)^n\cdot \binom{m+n-1}{n}
  =n^n.
 \end{align*}
\end{proof}

\begin{defn}\begin{enumerate}[label=(\alph*)]
\item  A flat morphism $\pi:(\cX,\Delta) \to C$ over a smooth curve $C$ together with a section $\sigma: C\to\cX$ is called a \emph{$\mathbb{Q}$-Gorenstein flat family of klt singularities} if it satisfies the following conditions:
 \begin{itemize}
     \item $\cX$ is normal,  $\Delta$ is an effective $\bQ$-divisor on $\cX$,
     and $K_{\mathcal{X}}+\Delta$ is $\mathbb{Q}$-Cartier;
     \item For any $c\in C$, the fiber $\cX_c$ is normal and not contained in $\Supp(\Delta)$;
     \item $(\cX_c,\Delta_c)$ is a klt pair for any closed point $c\in C$.
 \end{itemize}
 
\item Given a $\mathbb{Q}$-Gorenstein  flat family of klt pairs $\pi:(\cX,\Delta)\to C$ and $\sigma:C\to\cX$, we call a proper birational morphism $\mu: \cY\to\cX$ provides a \emph{flat family of Koll\'ar components} $S$ over $(\cX,\Delta)$ (centered at $\sigma(C)$) if the following conditions hold:
 \begin{itemize}
     \item $\cY$ is normal, $\mu$ is an isomorphism over $\cX\setminus\sigma(C)$ and $S=\Ex(\mu)$ is a prime divisor on $\cY$;
     \item $-S$ is $\bQ$-Cartier and $\mu$-ample;
     \item $S_c:=S|_{\cY_c}$ is a prime divisor on $\cY_c$ which 
     gives Koll\'ar component fiberwisely, i.e.,  $(\cY_c, S_c+(\mu_*^{-1}\Delta)|_{\cY_c})$ is a plt pair for any closed point $c\in C$.
 \end{itemize}
 \end{enumerate}
\end{defn}

\begin{lem}\label{l-kollardeg}
Let $\pi:(\cX,\Delta)\to C$ be a $\bQ$-Gorenstein flat family of klt singularities over a smooth curve $C$ with a section $\sigma\colon C\to X$, such that $c\mapsto\hvol(\sigma(c),\cX_c, \Delta_c)$ is constant.
Let $0\in C$ be a closed point. Denote by $C^{\circ}:=C\setminus\{0\}$, $\cX^{\circ}:=\pi^{-1}(C^{\circ})$ and $\Delta^{\circ}:=\Delta|_{\cX^{\circ}}$.
Suppose there exists a proper birational morphism $\mu^{\circ}:\cY^{\circ}\to\cX^{\circ}$ which provides a flat family of Koll\'ar components $S^{\circ}$ over $(\cX^\circ,\Delta^{\circ})$.

If $S^{\circ}_c$ computes $\hvol(\sigma(c),\cX_c,\Delta_c)$ for all $c\in C^{\circ}$, then there exists a proper birational morphism $\mu:\cY\to\cX$ as an extension of $\mu^{\circ}$ which provides a flat family of Koll\'ar components $S$ over $(\cX,\Delta)$, such that $S_{0}$ computes $\hvol(\sigma(0),\cX_{0},\Delta_{0})$.
\end{lem}

\begin{proof}

Let us fix $k$ sufficiently divisible so that 
$k S^{\circ}$ is Cartier. Since $-S^{\circ}$ is ample over $\cX$, after replacing $k$ again, we can assume that
 $$\mu^{\circ}_*\cO_{\cY^\circ}(-kmS^{\circ})=_{\rm defn}\fb_{k m}^{\circ}=(\fb_{k}^{\circ})^m$$ is a flat family of ideals over $C^{\circ}$ for any $m\in\bZ_{>0}$.
Then there exists a unique ideal sheaf $\fb_{km}$ on $\cX$ which extends $\fb_{km}^{\circ}$ to a flat family of ideals over $C$. By the same reason as argued in the proof of Lemma \ref{hvolbound-ineq}, $(\fb_{k\bullet})$ is a graded sequence of flat families of ideals over $C$.

Denote by $\alpha:=A_{(\cX^{\circ},\Delta^{\circ})}(S^{\circ})$. 
By adjunction, it is clear that $\alpha=A_{(\cX_c,\Delta_c)}(S_c^{\circ})$ for any $c\in C^{\circ}$. 
Since $-kS^\circ$ is Cartier on $\cY^{\circ}$ and relatively ample
over $\cX^{\circ}$, we have $(\mu^\circ)^{-1}\fb_{km}\cdot\cO_{\cY^{\circ}}=\cO_{\cY}(-kmS^{\circ})$.
By inversion of adjunction, $(\cY^\circ,\mu_*^{-1}\Delta^\circ+ S^\circ)$ is plt.
As a result, by pulling back to $\cY^{\circ}$ and $\cY_c$ we get \[
\lct(\cX^{\circ},\Delta^{\circ};\fb_{km}^{\circ})=\lct(\cX_c,\Delta_c;\fb_{c,km})=\frac{\alpha}{km}.
\]
Since $\hvol(\sigma(0),\cX_{0},\Delta_0)=\hvol(\sigma(c),\cX_c,\Delta_c)$, we know that $\lim_{m\to\infty}km\cdot\lct(\cX_{0},\Delta_{0};\fb_{0,km})=\alpha$. Hence given $\epsilon>0$ sufficiently small, we have
$\lct(\cX_{0},\Delta_{0};\fb_{0,km})>\frac{\alpha-\epsilon}{km}$ for $m\gg 0$. By adjunction, we know that $(\cX,\Delta+\frac{\alpha-\epsilon}{km}\fb_{mk})$
is klt, and $a(S^{\circ};\cX,\Delta+\frac{\alpha-\epsilon}{km}\fb_{mk})<0$. 

 Then by \cite[1.4.3]{BCHM10}, we can construct a relative projective model
$\mu\colon \cY\to (\cX,\Delta)$ extending $\mu^{\circ}$, such that the exceptional locus of $\mu$ is precisely the prime divisor $S$
and $-S$ is $\mu$-nef.
Denote by $S_0=\sum_{i} m_i S_0^{(i)}$ where $S_0^{(i)}$ are irreducible
components of $S_0$, then by adjunction we have
\[
 K_{\cY_{0}}+(\mu_0)_*^{-1}(\Delta_0)\sim_{\bQ}(K_{\cY}+\mu_*^{-1}\Delta+\cY_0)|_{\cY_0}\sim_{\bQ}\mu_0^*(K_{\cX_0}+\Delta_0)
 +(\alpha-1) S_0.
\]
Hence $A_{(\cX_0,\Delta_0)}(S_0^{(i)})= m_i(\alpha-1)+1\leq m_i\alpha$.
Therefore by \cite[Section 3.1]{LX16}, for any $c\in C^\circ$ we have
\begin{align*}
 \hvol(\sigma(0),\cX_0,\Delta_0)& \leq 
 \hvol(\cY_0/(\cX_0,\Delta_0)) = \vol_{\sigma(0)}^F(-\sum_{i}
 A_{(\cX_0,\Delta_0)}(S_0^{(i)}) S_0^{(i)})\\
 & \leq \vol_{\sigma(0)}^F(-\alpha\sum_i m_i S_0^{(i)}) =\vol_{\sigma(0)}^F(-\alpha S_0)
 = \alpha^n\cdot(-(-S_0)^{n-1})\numberthis\label{eq_hvolequal}
 \\&=\alpha^n\cdot (-(-S_c)^{n-1})
 =\hvol(\sigma(c),\cX_c,\Delta_c).
\end{align*}
Since $\hvol(\sigma(0),\cX_0,\Delta_0)=\hvol(\sigma(c),\cX_c,\Delta_c)$
by assumption, we conclude the two inequalities in \eqref{eq_hvolequal} have to
be equalities. The first inequality being equality
means that the model $\cY_0/(\cX_0,\Delta_0)$
computes the volume $\hvol(\sigma(0),\cX_0,\Delta_0)$, 
so it must be a model extracting a Koll\'ar
component $S_0^{(1)}$ by \cite[Proof of Theorem C]{LX16}; 
the second inequality being equality implies 
$A_{(\cX_0,\Delta_0)}(S_0^{(1)})=m_1\alpha$, so $S_0=S_0^{(1)}$
is reduced. Hence we finish the proof.
\end{proof}

The following result implies Theorem \ref{t-hvolbound} by setting $\Delta=0$.

\begin{thm}\label{t-loghvolbound}
 Let $x\in (X,\Delta)$ be an $n$-dimensional klt singularity.
 Then $\hvol(x,X,\Delta)\leq n^n$ and the equality holds if
 and only if $x\in X\setminus\Supp(\Delta)$ is smooth.
\end{thm}

\begin{proof}
The inequality case is in Lemma \ref{hvolbound-ineq}. For the equality 
case, let us assume that $\tau\colon C\to X$ such that $C$ is
a smooth curve, $\tau(C^{\circ}=_{\rm defn}C\setminus \{0\})$ is contained
in $X_{\rm reg}\setminus\Supp(\Delta)$ and $x=\tau(0)$ satisfies
$\hvol(x, X,\Delta)=n^n$.

By Lemma \ref{l-kollardeg}, we know the model $\cY^{\circ} $
obtained by the standard blow up along the image 
$\sigma^\circ=(\tau^\circ,\id)\colon C^{\circ} \to X_{\rm reg}\times C^{\circ}$, 
degenerates to a model $\mu:\cY\to(X\times C,\Delta\times C)$ which provides a flat
family of Koll\'ar component $S$. In particular,  over the special fiber $0$,
we obtain a model $\mu_0:\cY_0\to(X,\Delta)$  which yields a Koll\'ar component $S_0$ over $x$. 
Denote by $\Gamma_0$ the different of $(\mu_0)_*^{-1}\Delta$ on $S_0$.
Since $S_0$ computes $\hvol(x,X,\Delta)$ by Lemma \ref{l-kollardeg}, we have
$$n^n=A_{(X,\Delta)}(S_0) \cdot (-K_{S_0}-\Gamma_0)^{n-1}
=n\cdot (-K_{S_0}-\Gamma_0)^{n-1}.$$
By \cite[Theorem D]{LX16}, $(S_0,\Gamma_0)$ is log K-semistable  with
volume $ n^{n-1}$. If $\Gamma_0\neq 0$, we pick a closed
point $y\in (S_0)_{\rm reg}\cap \Supp(\Gamma_0)$.
It is easy to see that $A_{(S_0,\Gamma_0)}(\ord_y)< n-1$ and 
$\vol_{y,S_0}(\ord_y)=1$, hence \cite[Proposition 4.6]{LL19} implies
$$n^{n-1}=(-K_{S_0}-\Gamma_0)^{n-1}\leq
\left(\frac{n}{n-1}\right)^{n-1}A_{(S_0,\Gamma_0)}(\ord_y)^{n-1}\cdot\vol_{y,S_0}(\ord_y)<n^{n-1}$$
and we get a contradiction. So $\Gamma_0=0$ and hence $(S_0,\Gamma_0)\cong(\mathbb{P}^{n-1},0)$ by
\cite[Theorem 36]{Liu16}. And $-S|_{S_t}$ gives $\mathcal{O}(1)$ fiberwisely, thus $x\in X$ is smooth, as $\cY_0\to X$ induces a degeneration of $x\in X$ 
to $C(\mathbb{P}^{n-1}, \mathcal{O}(1))$ which is smooth. Since
$x\in X$ is smooth and the different $\Gamma_0=0$, we have $x\not\in\Supp(\Delta)$. Thus we finish 
the proof.
\end{proof}

\end{document}